\documentclass[a4paper]{amsart}

\usepackage{hyperref}
 
\usepackage{amsmath,amssymb,microtype}

\usepackage{todonotes}
\usepackage{amsthm}
\usepackage{amssymb}
\usepackage[nameinlink,capitalize]{cleveref}
\usepackage{xparse}

\usepackage[inline]{enumitem}

\title{Knapsack problems for wreath products}

\author[M.~Ganardi]{Moses Ganardi}
\author[D.~K\"onig]{Daniel K\"onig}
\author[M.~Lohrey]{Markus Lohrey}
\address{Universit\"at Siegen, Germany\\
  \texttt{\{ganardi,koenig,lohrey\}@eti.uni-siegen.de}}
\author[G.~Zetzsche]{Georg Zetzsche}
\address{LSV, CNRS \& ENS Paris-Saclay, France\\
\texttt{zetzsche@lsv.fr}}
 \thanks{The fourth author is supported by a fellowship within the Postdoc-Program of the German Academic Exchange Service (DAAD) and by
Labex DigiCosme, Univ.\ Paris-Saclay, project VERICONISS.}

\subjclass{20F10}
\keywords{knapsack, wreath products, decision problems in group theory}

\theoremstyle{plain}
\newtheorem{thm}{Theorem}[section]
\newtheorem{prop}[thm]{Proposition}
\newtheorem{lem}[thm]{Lemma}
\newtheorem{cor}[thm]{Corollary}

\theoremstyle{definition}
\newtheorem{ex}[thm]{Example}

\Crefname{thm}{Theorem}{Theorems}
\Crefname{lem}{Lemma}{Lemmas}
\Crefname{prop}{Proposition}{Propositions}
\Crefname{cor}{Corollary}{Corollaries}
\Crefname{ex}{Example}{Examples}
\Crefname{rem}{Remark}{Remarks}

\DeclareMathOperator{\id}{id}
\DeclareMathOperator{\ord}{ord}
\newcommand{\Knapsack}{\textsc{KP}}
\newcommand{\ExpEq}{\textsc{ExpEq}}

\newcommand{\MEM}{\textsc{Membership}}
\newcommand{\NP}{\mathsf{NP}} 
\newcommand{\TC}{\mathsf{TC}}
\newcommand{\AC}{\mathsf{AC}}
\newcommand{\WP}{\textsc{WP}}

\newcommand{\supp}{\mathsf{supp}} 
\newcommand{\N}{\mathbb{N}}
\newcommand{\Z}{\mathbb{Z}}
\newcommand{\DHB}{H_3(\Z)}
\newcommand{\rest}{\mathord\restriction}
\newcommand{\Powerset}[1]{\mathcal{P}(#1)}
\newcommand{\Sol}{\mathsf{Sol}}
\newcommand{\support}{\mathsf{supp}}

\DeclareDocumentCommand{\orderprod}{O{} m}{%
  \mathchoice{
    \sideset{}{^{#1}}\prod_{#2}
  }{
    \prod_{#2}^{#1}
  }{
    \prod_{#2}^{#1}
  }{
    \prod_{#2}^{#1}
  }
}

\newcommand{\varx}[1]{x_{#1}} 
\newcommand{\varCompare}[1]{z_{#1}} 
\newcommand{\varMove}[1]{y_{#1}}

\usepackage{array}

\makeatletter
\newcommand{\thickhline}{%
    \noalign {\ifnum 0=`}\fi \hrule height 1pt
    \futurelet \reserved@a \@xhline
}
\newcolumntype{"}{@{\hskip\tabcolsep\vrule width 1pt\hskip\tabcolsep}}
\makeatother

\begin{document}

\maketitle

\begin{abstract}
In recent years, knapsack problems for (in general non-commutative) groups have attracted attention.
In this paper, the knapsack problem for wreath products is studied. It turns out that decidability of knapsack 
is not preserved under wreath product. On the other hand, the class of knapsack-semilinear groups, where solutions
sets of knapsack equations are effectively semilinear, is closed under wreath product.  As a consequence, we obtain the decidability of knapsack for free solvable groups. Finally, it is shown that for every non-trivial abelian group $G$, knapsack (as well as the related subset sum problem)
for the wreath product $G \wr \Z$ is $\NP$-complete.
\end{abstract}

\section{Introduction}

In \cite{MyNiUs14}, Myasnikov, Nikolaev, and Ushakov began  the investigation of classical 
discrete optimization problems, which are formulated over the integers,
for arbitrary (possibly non-commutative) groups. The general goal of this line of research is to study 
to what extent results from the commutative setting can be transferred to the non-commutative setting.
Among other problems, Myasnikov et al.~introduced for a finitely generated group $G$ 
the {\em knapsack problem} and the {\em subset sum problem}.
The input for the knapsack problem is a sequence of group elements $g_1, \ldots, g_k, g \in G$ (specified
by finite words over the generators of $G$) and it is asked whether there exists a solution
$(x_1, \ldots, x_k) \in \mathbb{N}^k$
of the equation $g_1^{x_1} \cdots g_k^{x_k} = g$. For the subset sum problem one restricts the solution
to $\{0,1\}^k$. For the particular case $G = \mathbb{Z}$  (where the additive notation 
$x_1 \cdot g_1 + \cdots + x_k \cdot g_k = g$ is usually preferred)
these problems are {\sf NP}-complete (resp., $\mathsf{TC}^0$-complete)
if the numbers $g_1, \ldots, g_k,g$ are 
encoded in binary representation \cite{Karp72,Haa11} (resp., unary notation
\cite{ElberfeldJT11}).

Another motivation is that decidability of knapsack for a group $G$ implies that the membership 
problem for polycyclic subgroups of $G$ is decidable. This follows from the well-known fact
that every polycyclic group $A$ has a generating set $\{a_1,\ldots, a_k\}$ such that every
element of $A$ can be written as $a_1^{n_1} \cdots a_k^{n_k}$ for $n_1, \ldots, n_k \in \N$, see e.g. \cite[Chapter~9]{Sims}.

In \cite{MyNiUs14}, Myasnikov et al.~encode elements of the finitely generated group $G$ by words over the group generators
and their inverses, which corresponds to the unary encoding of integers. There is also an encoding of words 
that corresponds to the binary encoding of integers, so called straight-line programs, and knapsack problems under this
encodings have been studied in \cite{LohreyZetzsche2016a}. In this paper, we only consider the case where input words are explicitly represented.
Here is a (non-complete) list of known results concerning knapsack and subset sum problems:
\begin{itemize}
\item Subset sum and knapsack can be solved in polynomial
time for every hyperbolic group \cite{MyNiUs14}. In \cite{FrenkelNU15} this result was extended to free products of any
number of hyperbolic groups and finitely generated abelian groups.
\item For every virtually nilpotent group, subset sum belongs to {\sf NL} (nondeterministic logspace) \cite{KoenigLohreyZetzsche2015a}.
On the other hand, there are nilpotent groups of class $2$ for which knapsack is undecidable. Concrete examples
are direct products of sufficiently many copies of the discrete Heisenberg group $H_3(\mathbb{Z})$ \cite{KoenigLohreyZetzsche2015a},
and free nilpotent groups of class $2$ and sufficiently high rank \cite{MiTr17}.
\item Knapsack for the discrete Heisenberg group $H_3(\mathbb{Z})$ is decidable \cite{KoenigLohreyZetzsche2015a}. In particular, 
together with the previous point it follows that decidability
of knapsack is not preserved under direct products.
\item For the following groups, subset sum is {\sf NP}-complete (whereas the word problem can be solved in polynomial time):
free metabelian non-abelian groups of finite rank, the wreath product $\mathbb{Z} \wr \mathbb{Z}$, Thompson's group $F$,
the Baumslag-Solitar group $\mathrm{BS}(1,2)$ \cite{MyNiUs14}, and every polycyclic group that is not virtually nilpotent  \cite{NiUs16}.
\item Knapsack is decidable for every co-context-free group (a group is co-context-free if the set of all words over the generators
that do not represent the group identity is a context-free language) \cite{KoenigLohreyZetzsche2015a}.
\item Knapsack belongs to $\NP$ for every virtually special group \cite{LohreyZetzsche2016a}. 
A group is virtually special if it is a finite extension of a subgroup
of a graph group. For graph groups (also known as right-angled Artin groups) a complete classification of the complexity
of knapsack was obtained in \cite{LohreyZ17}: If the underlying graph contains an induced path or cycle on 4 nodes, then knapsack
is $\NP$-complete; in all other cases knapsack can be solved in polynomial time (even in {\sf LogCFL}).
\item Decidability of knapsack is preserved under finite extensions, HNN-extensions over finite associated subgroups 
and amalgamated free products over finite subgroups \cite{LohreyZetzsche2016a}.
\end{itemize}
In this paper, we study the knapsack problem for wreath products.  
The wreath product is a fundamental construction in
group theory and semigroup theory, see \cref{sec-wreath} for the definition.
An important application of wreath products in group theory 
is the Magnus embedding theorem \cite{Mag39}, which allows to embed the 
quotient group $F_k/[N,N]$ into the wreath product
$\mathbb{Z}^k \wr (F_k/N)$, where $F_k$ is a free group of rank $k$ and $N$ is a normal
subgroup of $F_k$. 
From the algorithmic point of view, wreath products have some nice properties: The word problem
for a wreath product $G \wr H$ is $\AC^0$-reducible to the word problems for the factors $G$ and $H$,
and the conjugacy problem for $G \wr H$ is $\TC^0$-reducible to the conjugacy problems for $G$ and $H$
and the so called power problem for $H$ \cite{MiasnikovVW17}. 

As in the case of direct products, it turns out that decidability of knapsack is not preserved under wreath products: For this
we consider direct products of the form $\DHB\times\Z^\ell$, where $\DHB$ is the discrete 3-dimensional Heisenberg group.
It was shown in \cite{KoenigLohreyZetzsche2015a} that for every $\ell \geq 0$, knapsack is decidable for $\DHB\times\Z^\ell$.
We prove in \cref{undecidability} that for every non-trivial group $G$ and every sufficiently large $\ell$, 
knapsack for $G \wr (\DHB\times\Z^\ell)$ is undecidable. 

By the above discussion, we need stronger assumptions on $G$ and $H$ to obtain decidability of knapsack for $G \wr H$.
We exhibit a very weak condition on $G$ and $H$, knapsack-semilinearity, which is sufficient for decidability of knapsack for $G\wr H$.
 A finitely generated group $G$ is knapsack-semilinear 
if for every knapsack equation, the set of all solutions (a solution can be seen as an vector of natural numbers) is effectively semilinear.

Clearly, for every knapsack-semilinear group, the knapsack problem is
decidable. While the converse is not true, the class of knapsack-semilinear groups is extraordinarily wide.
The simplest examples are finitely generated abelian groups, but 
it also includes the rich class of virtually special groups \cite{LohreyZetzsche2016a}, all hyperbolic groups (see~\cref{appendix-hyperbolic}), and all co-context-free groups~\cite{KoenigLohreyZetzsche2015a}. Furthermore, it is known to be closed under direct products (an easy observation), finite extensions, 
HNN-extensions over finite associated subgroups 
and amalgamated free products over finite subgroups  (the last three closure properties 
are simple extensions of the transfer theorems in~\cite{LohreyZetzsche2016a}).
In fact, the only non-knapsack-semilinear groups with a decidable knapsack problem that we are aware of 
are the groups $\DHB\times\Z^n$.

We prove in \cref{decidability} that the class of knapsack-semilinear groups is closed under wreath products.
As a direct consequence of the Magnus embedding, it follows that knapsack is decidable for every free solvable group. 
Recall, that in contrast, knapsack for free nilpotent groups is in general undecidable \cite{MiTr17}.
 
Finally, we consider the complexity of knapsack for wreath products. 
We prove that for every non-trivial finitely generated abelian group $G$,  knapsack for 
$G \wr \Z$ is  $\NP$-complete (the hard part is membership in $\NP$).
This result includes important special cases like for instance the lamplighter group $\Z_2 \wr \Z$ and $\Z \wr \Z$.
Wreath products of the form $G \wr \Z$ with $G$ abelian turn out to be important in connection with subgroup 
distortion \cite{DaOl11}.  Our proof also shows that for every non-trivial finitely generated abelian group $G$,
the subset sum problem for $G \wr \Z$ is $\NP$-complete. In \cite{MyNiUs14} this result is only shown 
for infinite abelian groups $G$.

\section{Preliminaries}

We assume standard notions concerning groups. A group $G$ is {\em
  finitely generated} if there exists a finite subset $\Sigma
\subseteq G$ such that every element $g \in G$ can be written as $g =
a_1 a_2 \cdots a_n$ with $a_1, a_2, \ldots, a_n \in \Sigma$.  We also
say that the word $a_1 a_2 \cdots a_n \in \Sigma^*$ evaluates to $g$
(or represents~$g$).  The set $\Sigma$ is called a finite generating
set of $G$. We always assume that $\Sigma$ is symmetric in the sense
that $a \in \Sigma$ implies $a^{-1} \in \Sigma$.  An element $g\in G$
is called \emph{torsion element} if there is an $n\ge 1$ with $g^n=1$.
The smallest such $n$ is the \emph{order} of $g$ and denoted
$\ord(g)$. If $g$ is not a torsion element, we set $\ord(g)=\infty$.

A set of vectors $A \subseteq \mathbb{N}^k$ is {\em linear} if there exist vectors
$v_0, \ldots, v_n \in \mathbb{N}^k$ such that 
\[
A = \{ v_0 + \lambda_1 \cdot v_1 + \cdots +  \lambda_n \cdot v_n \mid \lambda_1,\ldots,\lambda_n\in\N\} .
\]
The tuple of vectors $(v_0, \ldots, v_n)$ is a \emph{linear represention} of $A$.
A set $A \subseteq \mathbb{N}^k$ is {\em semilinear} if it is a finite union of linear sets $A_1, \ldots, A_m$.
A {\em semilinear representation} of $A$ is a list of linear representations for the linear sets $A_1, \ldots, A_m$.
It is well-known that the semilinear subsets of $\N^k$ 
are exactly the sets definable in {\em Presburger arithmetic}. These are 
those sets that can be defined with a first-order formula $\varphi(x_1, \ldots, x_k)$ over the structure
$(\N,0,+,\le)$~\cite{GinsburgSpanier1966}. Moreover, the transformations between such a first-order formula and an equivalent  semilinear  representation
are effective. In particular, the semilinear sets are effectively closed under Boolean operations.

\section{Knapsack for  groups}

Let $G$ be a finitely generated group with the finite symmetric generating set $\Sigma$.
Moreover, let $V$ be a set of formal variables that take values
from $\N$. For a subset $U\subseteq V$, we use $\N^U$ to denote the set of 
maps  $\nu \colon U \to \N$, which we call \emph{valuations}.
An \emph{exponent expression} over $G$ is a formal expression of the form $E = v_0 u_1^{x_1} v_1 u_2^{x_2} v_2  \cdots u_k^{x_k} v_k$
with $k \geq 0$ and words $u_i, v_i \in \Sigma^*$. Here, the variables do not have to be pairwise distinct. If every variable in an exponent expression occurs at most once, it is called a \emph{knapsack expression}. 
Let $V_E = \{ x_1, \ldots, x_k \}$ be the set of variables that occur in $E$.
 For a valuation $\nu\in\N^U$ such that $V_E \subseteq U$ (in which case we also say
 that $\nu$ is a valuation for $E$), we define 
 $\nu(E) = v_0 u_1^{\nu(x_1)} v_1 u_2^{\nu(x_2)} v_2  \cdots u_k^{\nu(x_k)} v_k \in \Sigma^*$.
We say that $\nu$ is a \emph{solution} of the equation $E=1$ if $\nu(E)$ evaluates to the identity element $1$ of $G$.
With $\Sol(E)$ we denote the set of all solutions $\nu \in \N^{V_E}$ of $E$. We can view $\Sol(E)$ as a subset
of $\N^k$.  The \emph{length} of $E$ is defined as $|E| = |v_0| + \sum_{i=1}^k |u_i|+|v_i|$, whereas $k$ is its \emph{depth}.
If the length of a knapsack expression is not needed, we will write an exponent expression over $G$ also
as $E = h_0 g_1^{x_1} h_1 g_2^{x_2} h_2  \cdots g_k^{x_k} h_k$ where $g_i, h_i \in G$. 
We define {\em solvability of exponent equations over $G$}, $\ExpEq(G)$ for short, 
as the following decision problem:
\begin{description}
\item[Input] A finite list of exponent expressions $E_1,\ldots,E_n$ over $G$.
\item[Question] Is $\bigcap_{i=1}^n \Sol(E_i)$ non-empty?
\end{description}
The knapsack problem for $G$, $\Knapsack(G)$ for short, is the following 
decision problem:
\begin{description}
\item[Input] A single knapsack expression $E$ over $G$.
\item[Question] Is $\Sol(E)$ non-empty?
\end{description}
We also consider the uniform knapsack problem for powers 
\[
G^m = \underbrace{G \times \cdots \times G}_{\text{$m$ many}} .
\]
We denote this problem with $\Knapsack(G^\ast)$. Formally, it is defined as follows:
\begin{description}
\item[Input] A number $m \geq 0$ (represented in unary notation) and a knapsack expression $E$ 
over the group $G^m$.
\item[Question] Is $\Sol(E)$ non-empty?
\end{description}
It turns out that the problems $\Knapsack(G^\ast)$ and $\ExpEq(G)$ are interreducible:
\begin{prop}\label{kppowers-vs-expeq}
  For every finitely generated group $G$, $\Knapsack(G^\ast)$ is decidable if and only if
  $\ExpEq(G)$ is decidable.
\end{prop}
\begin{proof}
  Clearly, every instance of $\Knapsack(G^\ast)$ can be translated to
  an instance of $\ExpEq(G)$ by projecting onto the $m$ factors of a power $G^m$. For the converse direction, assume that
  $\Knapsack(G^\ast)$ is decidable. Then in particular, $G$ has a
  decidable word problem. Let
  $E_j=h_{0,j}g_{1,j}^{\varx{1,j}}h_{1,j}\cdots
  g_{k,j}^{\varx{k,j}}h_{k,j}$ be an exponent expression over $G$ for
  every $j\in[1,m]$. By adding dummy powers of the form $1^x$  we may
  assume that the $E_j$ have the same depth $k$.
  We distinguish two cases.
  
  \medskip
  \noindent
  {\em Case 1.} $G$ is a torsion group.
  Since $G$ has a decidable word
    problem, we can compute $\ell\in\N$ so that $g_{i,j}^\ell=1$ for
    every $i\in[1,k]$ and $j\in[1,m]$. Then there is a solution to the
    exponent equation system if and only if there is a solution $\nu$
    with $0\le\nu(\varx{}) < \ell$ for every variable
    $\varx{}$. Hence, solvability is clearly decidable.

  \medskip
  \noindent
  {\em Case 2.}
   There is some $a\in G$ with $\ord(a)=\infty$. We first rename the variables
   in $E_1,\ldots,E_m$ such that every variable occurs at most once in the entire
   system of expressions. Let $E'_1, \ldots, E'_m$ be the resulting system of knapsack
   expressions and let $U$ be the set of variables that occur in $E'_1, \ldots, E'_m$. We can
   compute an  equivalence relation
    $\mathord{\sim}\subseteq U\times U$  such that the
    system $E_1=1,\ldots,E_m=1$ has a solution if and only if the system
    $E'_1=1,\ldots,E'_m=1$ has a solution $\nu$ with
    $\nu(\varx{})=\nu(\varx{}')$ for $\varx{}\sim\varx{}'$.  We can
    equip $U$ with a linear order $\le$ so that if $\varx{}$ occurs
    left of $\varx{}'$ in some $E'_j$, then $\varx{} < \varx{}'$.

    Now for each pair $(\varx{},\varx{}')\in U\times U$ with
    $\varx{}\sim \varx{}'$ and $\varx{}<\varx{}'$, we add the
    knapsack expression $a^{\varx{}}(a^{-1})^{\varx{}'}$. This yields knapsack expressions
    $E'_1,\ldots,E'_{m+\ell}$ for some $\ell\ge 0$ such that
    $E'_1=1,\ldots,E'_{m+\ell}=1$ is solvable if and only if
    $E_1=1,\ldots,E_m=1$ is solvable. Moreover, whenever $\varx{}$ occurs
    to the left of $\varx{}'$ in some expression, then
    $\varx{} < \varx{}'$.
    
    By padding the expressions with trivial powers, we turn
    $E'_1,\ldots,E'_{m+\ell}$ into expressions
    $E''_1,\ldots,E''_{m+\ell}$ that all exhibit the same variables
    (in the same order). Now, it is easy to turn
    $E''_1,\ldots,E''_{m+\ell}$ into a single knapsack expression over
    $G^{m+\ell}$.
\end{proof}
Note that the equation $v_0 u_1^{x_1} v_1 u_2^{x_2} v_2  \cdots u_k^{x_k} v_k=1$ is equivalent 
to 
\[
(v_0 u_1 v_0^{-1})^{x_1} (v_0 v_1 u_2 v_1^{-1} v_0^{-1})^{x_2} \cdots (v_0 \cdots v_{k-1} u_k v_{k-1}^{-1} \cdots v_0^{-1})^{x_k}
(v_0 \cdots v_k) = 1 .
\]
Hence, it suffices to consider exponent expressions of the form
$u_1^{x_1} u_2^{x_2} \cdots u_k^{x_k} v$.

The group $G$ is called {\em knapsack-semilinear} if for every knapsack expression $E$ over $G$,
the set $\Sol(E)$ is a semilinear set of vectors and a semilinear representation can be effectively computed from $E$.
The following classes of groups only contain knapsack-semilinear groups:
\begin{itemize}
\item virtually special groups~\cite{DBLP:journals/corr/LohreyZ15}: these are finite extensions of subgroups of graph groups (aka right-angled
Artin groups). The class of virtually special groups is very rich. It contains all Coxeter groups, 
one-relator groups with torsion, fully residually free groups, and  fundamental groups of hyperbolic 3-manifolds.
\item hyperbolic groups: see \cref{appendix-hyperbolic}
\item co-context-free groups~\cite{KoenigLohreyZetzsche2015a}, i.e., groups where the set of all words over the generators that do not represent the identity
is a context-free language. Lehnert and Schweitzer \cite{LehSch07} have shown that the Higman-Thompson groups are
co-context-free.
\end{itemize}
Since the emptiness of the intersection of finitely many semilinear sets is decidable, we have:
\begin{lem}
If $G$ is knapsack-semilinear, then $\Knapsack(G^\ast)$ and $\ExpEq(G)$ are decidable.
\end{lem}
An example of a group $G$, where $\Knapsack(G)$ is decidable but $\Knapsack(G^\ast)$ (and hence $\ExpEq(G)$)
are undecidable is the Heisenberg group $\DHB$, see \cite{KoenigLohreyZetzsche2015a}. It is the group of all matrices of the following form, where $a,b,c \in \Z$:
\[
\left( \!\!
\begin{array}{ccc}
1 & a & c \\
0 & 1 & b \\
0 & 0 & 1
\end{array}
\!\! \right)
\]
In particular, $\DHB$ is not knapsack-semilinear.

\section{Wreath products} \label{sec-wreath}

Let $G$ and $H$ be groups. Consider the direct sum $K = \bigoplus_{h
  \in H} G_h$, where $G_h$ is a copy of $G$. We view $K$ as the set $G^{(H)}$ of
all mappings $f\colon H\to G$ such that $\support(f)=\{h\in H \mid f(h)\ne
1\}$ is finite, together with pointwise multiplication as the group
operation.  The set $\support(f)\subseteq H$ is called the
\emph{support} of $f$. The group $H$ has a natural left action on
$G^{(H)}$ given by $h f(a) = f(h^{-1}a)$, where $f \in G^{(H)}$ and
$h, a \in H$.  The corresponding semidirect product $G^{(H)} \rtimes
H$ is the \emph{wreath product} $G \wr H$.  In other words:
\begin{itemize}
\item
Elements of $G \wr H$ are pairs $(f,h)$, where $h \in H$ and
$f \in G^{(H)}$.
\item
The multiplication in $G \wr H$ is defined as follows:
Let $(f_1,h_1), (f_2,h_2) \in G \wr H$. Then
$(f_1,h_1)(f_2,h_2) = (f, h_1h_2)$, where
$f(a) = f_1(a)f_2(h_1^{-1}a)$.
\end{itemize}
The following intuition might be helpful:
An element $(f,h) \in G\wr H$ can be thought of
as a finite multiset of elements of $G \setminus\{1_G\}$ that are sitting at certain
elements of $H$ (the mapping $f$) together with the distinguished
element $h \in H$, which can be thought of as a cursor
moving in $H$.
If we want to compute the product $(f_1,h_1) (f_2,h_2)$, we do this
as follows: First, we shift the finite collection of $G$-elements that
corresponds to the mapping $f_2$ by $h_1$: If the element $g \in G\setminus\{1_G\}$ is
sitting at $a \in H$ (i.e., $f_2(a)=g$), then we remove $g$ from $a$ and
put it to the new location $h_1a \in H$. This new collection
corresponds to the mapping $f'_2 \colon  a \mapsto f_2(h_1^{-1}a)$.
After this shift, we multiply the two collections of $G$-elements
pointwise: If in $a \in H$ the elements $g_1$ and $g_2$ are sitting
(i.e., $f_1(a)=g_1$ and $f'_2(a)=g_2$), then we put the product
$g_1g_2$ into the location $a$. Finally, the new distinguished
$H$-element (the new cursor position) becomes $h_1 h_2$.

By identifying $f\in G^{(H)}$ with $(f,1_H)\in G\wr H$ and $h\in H$
with $(1_{G^{(H)}}, h)$, we regard $G^{(H)}$ and $H$ as subgroups of $G\wr H$.
Hence, for $f\in G^{(H)}$ and $h\in H$, we have $f h=(f,
1_H)(1_{G^{(H)}}, h)=(f, h)$.
There are two natural projection morphism 
$\sigma_{G\wr H}  \colon G\wr H\to H$  and $\tau_{G\wr H}  \colon G\wr G^{(H)}$
with
\begin{eqnarray} 
\sigma_{G\wr H}(f,h) & = & h,  \label{eq-proj} \\
\tau_{G\wr H}(f,h) & = & f . \label{eq-tau}
\end{eqnarray}
If $G$ (resp. $H$) is generated by the set $\Sigma$ (resp. $\Gamma$) with 
$\Sigma \cap \Gamma = \emptyset$, then 
$G \wr H$ is generated by the set 
$\{ (f_a,1_H) \mid a \in \Sigma \} \cup \{ (f_{1_G}, b) \mid b \in \Gamma\}$, where for $g \in G$,
the mapping $f_g : H \to G$ is defined by $f_g(1_H) = g$ and $f_g(x) = 1_G$ for $x \in H \setminus\{1_H\}$.
This generating set can be identified with 
$\Sigma \uplus \Gamma$.
We will need the following embedding lemma:
\begin{lem}\label{embedding-powers}
  Let $G,H,K$ be finitely generated groups where $K$ has a
  decidable word problem. Then, given $n\in\N$ with $n\le |K|$, one can
  compute an embedding of $G^n\wr H$ into $G\wr (H\times K)$.
\end{lem}

\begin{proof}
  Let $\Sigma$, $\Gamma$, and $\Theta$ be finite
  generating sets of $G$, $H$, and $K$, respectively.  Suppose
  $n\in\N$ is given.  Since $K$ has a decidable word problem and
  $|K| \geq n$, we can compute words
  $w_1,\ldots,w_n\in \Theta^*$ that represent pairwise
  distinct elements $k_1,\ldots,k_n$ of $K$.

  Let $\pi_i\colon G^n\to G$ be the projection on the $i$-th
  coordinate. Since the statement of the \lcnamecref{embedding-powers} does
  not depend on the chosen generating sets of $G^n\wr H$ and $G\wr
  (H\times K)$, we may choose one. The group $G^n$ is generated by the tuples
  $s_i := (1,\ldots,1,s,1,\ldots,1)\in G^n$, for $s\in\Sigma$ and
  $i\in[1,n]$, where $s$ is at the $i$-th coordinate. 
  Hence,  $\Delta=\{ s_i \mid s \in \Sigma, i \in [1,n] \}  \uplus \Gamma$ is a 
  finite  generating set of $G^n\wr H$.
    
  The embedding $\iota\colon \Delta^* \to (\Sigma\cup
  \Gamma\cup \Theta)^*$ is defined by
  $\iota(s_i)=w_isw_i^{-1}$ for $s\in\Sigma$, $i \in [1,n]$ and
  $\iota(t)=t$ for $t\in\Gamma$. It remains to be shown that $\iota$
  induces an embedding of $G^n\wr H$ into $G\wr(H\times K)$.

  Consider the injective morphism $\varphi\colon (G^n)^{(H)}\to G^{(H\times K)}$
  where for $\zeta\in (G^n)^{(H)}$, we have
  \[ 
  [\varphi(\zeta)](h,k)=
  \begin{cases}
  \pi_i(\zeta(h)) & \text{ if } k = k_i \\
  1 & \text{ if } k\notin\{k_1,\ldots,k_n\}
  \end{cases}
  \]
  We claim that $\varphi$ extends to an injective morphism $\hat{\varphi}\colon
(G^n)^{(H)}\rtimes H\to G^{(H\times K)}\rtimes H$ where $H$ acts on $G^{(H\times K)}$
by $(h \zeta) (a,k) = \zeta(h^{-1}a,k)$ for $h,a \in H$, $k \in K$. To show this, it suffices 
to establish $\varphi(h\zeta)=h\varphi(\zeta)$ for all $\zeta\in (G^n)^{(H)}$, $h \in H$, i.e.,
the action of $H$ commutes with the morphism $\varphi$. To see this, note that 
\[ [\varphi(h \zeta)](a,k_i)=\pi_i((h\zeta)(a))=\pi_i(\zeta(h^{-1}a))=[\varphi(\zeta)](h^{-1}a, k_i)  =  [h\varphi(\zeta)](a, k_i) \]
and  if $k\notin\{k_1,\ldots,k_n\}$, we have
\[ [\varphi(h\zeta)](a,k)=1=[\varphi(\zeta)](h^{-1}a, k) =  [h\varphi(\zeta)](a, k). \]
Since the above action of $H$ on $G^{(H\times K)}$ is the restriction of the action 
of $H \times K$ on $G^{(H\times K)}$, we have $G^{(H\times K)}\rtimes H \leq G^{(H\times K)}\rtimes (H \times K) 
= G \wr (H\times K)$. Thus $\hat{\varphi}$ can be viewed as an embedding $\hat{\varphi} : G^n\wr H \to G\wr(H\times K)$.

We complete the
proof by showing that $\iota$ represents $\hat{\varphi}$, i.e.
$\hat{\varphi}(\overline{w})=\overline{\iota(w)}$ for every $w\in \Delta^*$,
where $\overline{w}$ denotes the element of $(G^n)^{(H)}\rtimes H$ represented
by the word $w$ and similarly for $\iota(w)$. It
suffices to prove this in the case $w\in\Delta \subseteq (G^n)^{(H)}\rtimes H$. If
$w=s_i$ with $s \in\Sigma, i \in [1,n]$, we observe that
$\hat{\varphi}(s_i)=k_i s k_i^{-1}=\overline{\iota(s_i)}$. Moreover, for
$t\in\Gamma \subseteq H$ we have
$\hat{\varphi}(t)=t=\iota(t)$.
\end{proof}

\section{Main results}

In this \lcnamecref{undecidability}, we state the main results of the paper.
We begin with a general necessary condition for knapsack to
be decidable for a wreath product. Note that if $H$ is finite, then
$G\wr H$ is a finite extension of $G^{|H|}$~\cite[Proposition~1]{LohreySteinbergZetzsche2015a}, meaning that
$\Knapsack(G\wr H)$ is decidable if and only if $\Knapsack(G^{|H|})$
is decidable~\cite[Theorem 11]{LohreyZetzsche2016a}\footnote{Strictly
  speaking, only preservation of $\NP$-membership was shown
  there. However, the proof also yields preservation of
  decidability.}. Therefore, we are only interested in the case that
$H$ is infinite.

\begin{prop}\label{necessary-conditions}
  Suppose $H$ is infinite.  If
  $\Knapsack(G\wr H)$ is decidable, then $\Knapsack(H)$ and
  $\Knapsack(G^\ast)$ are decidable.
\end{prop}
\begin{proof}
  As a subgroup of $G\wr H$, $H$ inherits decidability of the knapsack
  problem. According to \cref{embedding-powers}, given $m\in\N$, we
  can compute an embedding of $G^m$ into $G\wr H$ and thus solve
  knapsack instances over $G^m$ uniformly in $m$.
\end{proof}
\cref{necessary-conditions} shows that $\Knapsack(\DHB\wr\Z)$ is undecidable: It was shown in
\cite{KoenigLohreyZetzsche2015a} that $\Knapsack(\DHB)$ is decidable,
whereas for some $m>1$, the problem $\Knapsack(\DHB^m)$ is
undecidable.

\cref{necessary-conditions}  raises the question whether decidability  of
$\Knapsack(H)$ and $\Knapsack(G^\ast)$ implies decidability of $\Knapsack(G\wr H)$.
The answer turns out to be negative. Let us first recall the following result from 
\cite{KoenigLohreyZetzsche2015a}: 

\begin{thm}[\cite{KoenigLohreyZetzsche2015a}]
For every $\ell\in\N$, $\Knapsack(\DHB\times\Z^\ell)$ is decidable.
\end{thm}
Hence, by the following result, which is shown in \cref{undecidability},  decidability  of
$\Knapsack(H)$ and $\Knapsack(G^\ast)$ does in general not imply decidability of $\Knapsack(G\wr H)$:

\begin{thm}\label{undecidability:general}
  There is an $\ell \in\N$ such that for every group $G \neq 1$, $\Knapsack(G\wr (\DHB\times\Z^\ell))$ is undecidable.
\end{thm}

We therefore need to strengthen the assumptions on $H$ in order to show decidability
of $\Knapsack(G\wr H)$. By adding the weak assumption of knapsack-semilinearity for $H$,
we obtain a partial converse to \cref{necessary-conditions}. In \cref{decidability} we prove:
\begin{thm} \label{thm-semilinear1}
  Let $H$ be knapsack-semilinear. Then $\Knapsack(G\wr H)$ is
  decidable if and only if $\Knapsack(G^\ast)$ is decidable.
\end{thm}
In fact, in case $G$ is also knapsack-semilinear, our algorithm constructs a semilinear representation of the solution set. Therefore, we get:

\begin{thm}  \label{thm-semilinear2}
  The group $G\wr H$ is knapsack-semilinear if and only if both $G$
  and $H$ are knapsack-semilinear.
\end{thm}
Since every free abelian group is clearly knapsack-semilinear, it follows that the iterated wreath products
$G_{1,r} = \mathbb{Z}^r$ and $G_{d+1,r} = \mathbb{Z}^r \wr G_{d,r}$ are knapsack-semilinear.
By the well-known Magnus embedding, the free solvable group $S_{d,r}$  embeds into 
$G_{d,r}$. Hence, we get:
\begin{cor}
Every free solvable group is knapsack-semilinear. Hence, solvability of exponent equations is decidable
for free solvable groups.
\end{cor}
Finally, we consider the complexity of knapsack for wreath products. We prove  $\NP$-completeness
for an important special case:

\begin{thm} \label{thm-NP}
For every non-trivial finitely generated abelian group $G$,  $\Knapsack(G \wr \Z)$ is  $\NP$-complete.
\end{thm}

\section{Undecidability: Proof of \cref{undecidability:general}}\label{undecidability}

Our proof of \cref{undecidability:general} employs the undecidability
of the knapsack problem for certain powers of $\DHB$.  In fact, we
need a slightly stronger version, which states undecidability already
for knapsack instances of bounded depths.
\begin{thm}[\cite{KoenigLohreyZetzsche2015a}] \label{thm:KoLoZe}
  There is a fixed constant $m$ and a fixed list of group elements
  $g_1,\ldots,g_k \in \DHB^m$ such that membership in the product
  $\prod_{i=1}^k \langle g_i\rangle$ is undecidable.  In
  particular, there are $k,m\in\N$ such that solvability of knapsack
  instances of depth $k$ is undecidable for $\DHB^m$.
\end{thm}
We prove \cref{undecidability:general} by showing the following.
\begin{prop}\label{undecidability:with-power}
  There are $m,\ell\in\N$ such that for every non-trivial group $G$, the
  knapsack problem for $G^m\wr (\DHB\times\Z^\ell)$ is undecidable.
\end{prop}
Let $k$ and $m$ be the constants from \cref{thm:KoLoZe}.
In order to prove \cref{undecidability:with-power},
consider a knapsack  expression
\begin{equation} E = g_1^{\varx{1}} \cdots  g_k^{\varx{k}} g_{k+1} \end{equation}
with $g_1,\ldots,g_{k+1} \in
\DHB^m$. We can write $g_i=(g_{i,1},\ldots,g_{i,m})$ for $i\in[1,k+1]$, which leads to the expressions
\begin{equation} E_j = g_{1,j}^{\varx{1,j}} \cdots g_{k,j}^{\varx{k,j}} g_{k+1,j}. \end{equation}
Let $\ell=m\cdot k$ and let $\alpha\colon \DHB\times\Z^\ell\to \DHB$
and $\beta\colon \DHB\times\Z^\ell\to\Z^\ell$ be the projection onto
the left and right component, respectively.  For each $p\in [1,\ell]$,
let $e_p\in\Z^\ell$ be the $p$-th unit vector
$e_p=(0,\ldots,0,1,0,\ldots,0)$.  For $j\in [1,m]$
we define the following knapsack expressions over $\DHB\times\Z^\ell$
($0$ denotes the zero vector  of dimension $\ell$):
\[
 E'_j = \prod_{i=1}^k
(g_{i,j},e_{(j-1)k+i})^{\varx{i,j}}  (g_{k+1,j},0) \quad
\text{ and } \quad  M_j = \prod_{t=1}^\ell (1,-e_{t})^{\varMove{j,t,0}}(1,e_{t})^{\varMove{j,t,1}} .
\]
Note that the term
$(j-1)k+i$ assumes all numbers $1,\ldots,m\cdot k$ as $i$ ranges over
$1,\ldots,k$ and $j$ ranges over $1,\ldots,m$. 

Since $G$ is non-trivial, there is some $a\in
G\setminus\{1\}$. For each $j\in[1,m]$, let
$a_j=(1,\ldots,1,a,1,\ldots,1)\in G^m$, where the $a$ is in the $j$-th
coordinate. With this, we define
\[
C = \prod_{i=1}^k \bigg(\prod_{j=1}^m (1,-e_{(j-1)k+i})\bigg)^{\varCompare{i}} \quad
\text{ and } \quad  F= \bigg( \prod_{j=1}^m a_j \, E'_j \bigg) \ C \ \bigg( \prod_{j=1}^{m}a_j^{-1} \, M_j \bigg).
\]
Since $G^m$ and $\DHB\times\Z^\ell$ are subgroups of $G^m \wr (\DHB\times\Z^\ell)$, we can treat
$F$ as a knapsack expression over $G^m \wr (\DHB\times\Z^\ell)$.
We will show that $\Sol(F) \neq \emptyset$  if and only if $\Sol(E) \neq \emptyset$. For this we need another
simple lemma:

\begin{lem}\label{no-movement}
  Let $G,H$ be groups and let $a\in G\setminus\{1\}$ and $f,g,h\in H$.
  Regard $G$ and $H$ as subsets of $G\wr H$. Then $faga^{-1}h=1$ if
  and only if $g=1$ and $fh=1$.
\end{lem}

\begin{proof}
  The right-to-left direction is trivial. For the converse, suppose
  $faga^{-1}h=1$ and $g\ne 1$.  By definition of $G\wr H$, we can
  write $faga^{-1}h=(\zeta, p)$ with $\zeta\in G^{(H)}$ and $p\in H$,
  where $\zeta(f)=a\ne 1$, $\zeta(fg)=a^{-1}\ne 1$, and $p=fgh$. This
  clearly implies $faga^{-1}h\ne 1$, a contradiction. Hence,
  $faga^{-1}h=1$ implies $g=1$ and thus $fh=1$.
\end{proof}
In the proof of the following lemma, we use the simple fact that
every morphism $\varphi\colon G\to G'$ extends uniquely to a morphism
$\hat{\varphi}\colon G\wr H\to G'\wr H$ such that
$\hat{\varphi}\rest_G=\varphi$ and $\hat{\varphi}\rest_H=\id_H$ (the identity mapping on $H$).

\begin{lem}\label{undecidability:solution:characterization}
A valuation $\nu$ for $F$ satisfies $\nu(F)=1$ if and only if
for every $i\in[1,k]$, $j\in[1,m]$, $t\in[1,m-1]$, we have
\begin{alignat}{2} \nu(E_j)&=1, \qquad & \nu(\varx{i,j})&=\nu(\varCompare{i}), \label{undecidability:solution:conditiona}\\
\nu(M_t)&=\nu(E'_t), \qquad & \nu(M_1\cdots M_m)&=1. \label{undecidability:solution:conditionb}
 \end{alignat}
\end{lem}
\begin{proof}
  Let $\pi_j\colon G^m\to G$ be the projection morphism onto the $j$-th
  coordinate and let $\hat{\pi}_j\colon G^m\wr (\DHB\times\Z^\ell)\to
  G\wr(\DHB\times \Z^\ell)$ be its extension with
  $\hat{\pi}_j\rest_{\DHB\times\Z^\ell}=\id_{\DHB\times\Z^\ell}$. 
  Of course, for $g\in G^m\wr(\DHB\times\Z^\ell)$, we have $g=1$ if and
  only if $\hat{\pi}_j(g)=1$ for every $j\in[1,m]$. 
  Observe that
     \[\hat{\pi}_r(\nu(F))=\nu\left( \bigg(\prod_{j=1}^{r-1} E'_j \bigg)
    \ a  \ \bigg(\prod_{j={r}}^m E'_j \bigg) \
    C \ \bigg( \prod_{j=1}^{r-1} M_j \bigg) \ a^{-1} \ \bigg( \prod_{j=r}^m
    M_j \bigg) \right)\] 
    for every $r\in[1,m]$.  Therefore, according to
  \cref{no-movement}, $\nu(F)=1$ holds if and only if for every
  $r\in[1,m]$, we have
\begin{equation} \nu(E'_1\cdots E'_mCM_1\cdots M_m)=1 \quad \text{ and } \quad \nu(E'_r\cdots E'_mCM_1\cdots M_{r-1})=1. \label{movements-neutral}
\end{equation}
We claim that  \cref{movements-neutral} holds for all $r\in[1,m]$ if and only if 
\begin{equation} \nu(E'_1\cdots E'_mC)=1, \quad \nu(E'_t)=\nu(M_t) \quad \text{ and } \quad \nu(M_1\cdots M_m)=1 \label{undecidability:solution:reformulation}
\end{equation}
for all $t\in[1,m-1]$. First assume that \cref{undecidability:solution:reformulation} holds for all  $t\in[1,m-1]$.
We clearly get $\nu(E'_1\cdots E'_mCM_1\cdots M_m)=1$ and $\nu(E'_r\cdots E'_mCM_1\cdots M_{r-1})=1$ 
for $r = 1$. The equations $\nu(E'_r\cdots E'_mCM_1\cdots M_{r-1})=1$ 
for $r \in [2,m]$ are obtained by conjugating $\nu(E'_1\cdots E'_mC)=1$ with 
$\nu(E'_1)=\nu(M_1), \ldots, \nu(E'_{r-1})=\nu(M_{r-1})$.
Now assume that \cref{movements-neutral} holds for all $r\in[1,m]$.
Taking $r=1$ yields $\nu(E'_1\cdots E'_mC)=1$ and hence $\nu(M_1\cdots M_m)=1$.
Moreover, we have
$\nu(E'_1\cdots E'_{r-1}) = \nu(E'_1 \cdots E'_mCM_1\cdots M_{r-1})= \nu(M_1\cdots M_{r-1})$
for all $r\in[1,m]$, which implies $\nu(E'_t)=\nu(M_t)$ for all $t\in[1,m-1]$.

Observe that by construction of $E'_j$ and $C$, we have
\begin{alignat}{2}
\alpha(\nu(E'_j))&=\nu(E_j), & \qquad \pi_{(j-1)k+i}(\beta(\nu(E'_1\cdots E'_m)))&=\nu(\varx{i,j}),\label{images-a}\\
\alpha(\nu(C))&=1, & \qquad \pi_{(j-1)k+i}(\beta(\nu(C)))&=-\nu(\varCompare{i}). \label{images-c}
\end{alignat}
for every $i\in[1,k]$ and $j\in[1,m]$.

Note that the equations in
\cref{undecidability:solution:reformulation} only involve elements of
$\DHB\times\Z^\ell$.  Since for elements $g\in \DHB\times\Z^\ell$, we have $g=1$ if and
only if $\alpha(g)=1$ and $\beta(g)=1$, the equation
$\nu(E'_1\cdots E'_mC)=1$ is equivalent to
$\alpha(\nu(E'_1\cdots E'_mC))=1$ and $\beta(\nu(E'_1\cdots
E'_mC))=1$. By \crefrange{images-a}{images-c}, this is equivalent to
$\nu(E_1\cdots E_m)=1$ and
$\nu(\varx{i,j})=\nu(\varCompare{i})$ for all $i\in[1,k]$ and
$j\in[1,m]$.  Finally, $\nu(E'_t)=\nu(M_t)$ implies
$\nu(E_t)=\alpha(\nu(E'_t))=\alpha(\nu(M_t))=1$ for all $t \in [1,m-1]$ and hence also $\nu(E_m)=1$. Thus,
\cref{undecidability:solution:reformulation} is equivalent to the
conditions in the
\lcnamecref{undecidability:solution:characterization}.
\end{proof}

\begin{lem}
$\Sol(F) \neq \emptyset$  if and only if $\Sol(E) \neq \emptyset$.
\end{lem}
\begin{proof}
  If $\nu(F)=1$, then according to
  \cref{undecidability:solution:characterization}, the valuation also
  satisfies $\nu(E_j)=1$ and
  $\nu(\varx{i,j})=\nu(\varCompare{i})$ for $i\in[1,k]$ and
  $j\in[1,m]$.  In particular $\nu(\varx{i,j})=\nu(\varx{i,j'})$
  for $j,j'\in[1,m]$. Thus, we have
  \[ g_1^{\nu(\varx{1,1})} \cdots g_k^{\nu(\varx{k,1})} g_{k+1}=1 \]
  and hence $\Sol(E) \neq \emptyset$.

  Suppose now that $\Sol(E) \neq \emptyset$. Then there is a valuation $\nu$ with
  $\nu(E_j)=1$ and $\nu(\varx{i,j})=\nu(\varx{i,j'})$ for
  $i\in[1,k]$ and $j,j'\in[1,m]$. We shall prove that we can
  extend $\nu$ so as to satisfy the conditions of
  \cref{undecidability:solution:characterization}.

  The left-hand equation in \cref{undecidability:solution:conditiona}
  is fulfilled already.  Since
  $\nu(\varx{i,j})=\nu(\varx{i,j'})$, setting
  $\nu(\varCompare{i})=\nu(\varx{i,1})$ will satisfy the
  right-hand equation of
  \cref{undecidability:solution:conditiona}. Finally, observe that by
  assigning suitable values to the variables $\varMove{j,s,b}$ for $j\in[1,m]$,
  $s\in[1,\ell]$, and $b\in\{0,1\}$, we can enforce any value from
  $\{1\}\times \Z^\ell$ for $\nu(M_j)$.  Therefore, we can extend
  $\nu$ so that it satisfies \cref{undecidability:solution:conditionb}
  as well.
\end{proof}
This completes the proof of \cref{undecidability:with-power}, which
allows us to prove \cref{undecidability:general}.
\begin{proof}[Proof of \cref{undecidability:general}]
  By \cref{undecidability:with-power}, there are $\ell,m\in\N$ such
  that the knapsack problem is undecidable for
  $G^m\wr(\DHB\times\Z^\ell)$.  According to \cref{embedding-powers},
  the group $G^m\wr(\DHB\times\Z^\ell)$ is a subgroup of
  $G\wr(\DHB\times\Z^{\ell+1})$, meaning that the latter also has an
  undecidable knapsack problem.
\end{proof}

\section{Decidability: Proof of \cref{thm-semilinear1} and \cref{thm-semilinear2}}\label{decidability}

Let us fix a wreath product $G \wr H$. Recall the projection homomorphisms $\sigma = \sigma_{G \wr H} \colon G \wr H \to H$ 
and $\tau = \tau_{G \wr H} \colon G \wr H \to G^{(H)}$ 
from \eqref{eq-proj}. For $g \in G \wr H$ we write $\supp(g)$ for $\supp(\tau(g))$.

A knapsack expression $E=h_0g^{x_1}_1h_1\cdots g^{x_k}_kh_k$ over
$G \wr H$ is called
\emph{torsion-free} if for each $i\in[1,k]$, either $\sigma(g_i)=1$ or
$\sigma(g_i)$ has infinite order.  A map $\varphi\colon\N^a\to\N^b$ is
called \emph{affine} if there is a matrix $A\in\N^{b\times a}$
and a vector $\mu\in\N^b$ such that $\varphi(\nu)=A\nu+\mu$ for every
$\nu\in\N^a$.  
\begin{prop}\label{torsion-free-instances}
  Let knapsack be decidable for $H$.
  For every knapsack expression $E$ over $G\wr H$, one can construct
  torsion-free expressions $E_1,\ldots,E_r$ and affine maps
  $\varphi_1,\ldots,\varphi_r$ such that $\Sol(E)=\bigcup_{i=1}^r \varphi_i(\Sol(E_i))$.
\end{prop}
\begin{proof}
  First of all, note that since knapsack is decidable for $H$, we can
  decide for which $i$ the element $\sigma(g_i)\in H$ has finite or
  infinite order.  For a knapsack expression $F=h_0g_1^{x_1}h_1\cdots
  g_k^{x_k}h_k$, let $t(F)$ be the set of indices of $i\in[1,k]$ such
  that $\sigma(g_i)\ne 1$ and $\sigma(g_i)$ has finite order.  We show
  that if $|t(E)|>0$, then one can construct expressions
  $E_0,\ldots,E_{r-1}$ and affine maps
  $\varphi_0,\ldots,\varphi_{r-1}$ such that $|t(E_j)|<|t(E)|$ and
  $\Sol(E)=\bigcup_{j=0}^{r-1} \varphi_j(\Sol(E_j))$. This suffices,
  since the composition of affine maps is again an affine
  map.

  Suppose $E=h_0g_1^{x_1}h_1\cdots g_k^{x_k}h_k$ and $\sigma(g_i)\ne
  1$ has finite order $r$. Note that we can compute $r$. For every
  $j\in[0,r-1]$, let
  \[ E_j = h_0 g_1^{x_1} h_1 \cdots g_{i-1}^{x_{i-1}} h_{i-1}
  (g_i^r)^{x_i} (g_i^jh_i) g_{i+1}^{x_{i+1}} h_{i+1}\cdots
  g_k^{x_k}h_k. \] 
  Let $X = \{ x_1, \ldots, x_r \}$.
  Moreover, let $\varphi\colon \N^X\to \N^X$ be the affine map
  such that for $\nu\in\N^X$, we have
  $\varphi_j(\nu)(x_\ell)=\nu(x_\ell)$ for $\ell\ne i$ and
  $\varphi_j(\nu)(x_i)=r\cdot\nu(x_i)+j$. Note that then 
  $\sigma(g_i^r)=\sigma(g_i)^r=1$ and thus
  $t(E_j)=t(E)\setminus\{i\}$. Furthermore, we clearly have
  $\Sol(E)=\bigcup_{j=0}^{r-1} \varphi_j(\Sol(E_j))$.
\end{proof}
Since the image of a semilinear set under an affine map is again
semilinear, \cref{torsion-free-instances} tells us that it suffices to
prove \cref{thm-semilinear1,thm-semilinear2} for torsion-free knapsack
expressions. For the rest of this section let us fix a torsion-free
knapsack expression $E$ over $G \wr H$. We can assume that
$E=g^{x_1}_1g_2^{x_2} \cdots g^{x_k}_k g_{k+1}$ (note that if $g$ has
infinite order than also $c^{-1} g c$ has infinite order).  We
partition the set $V_E=\{x_1,\ldots,x_k\}$ of variables in $E$ as
$V_E=S\uplus M$, where $S=\{x_i\in V_E \mid \sigma(g_i)=1\}$ and
$M=\{x_i\in V_E \mid \ord(\sigma(g_i))=\infty\}$.  In this situation,
the following notation will be useful. If $U=A\uplus B$ for a set of
variables $U \subseteq V$ and $\mu\in\N^A$ and $\kappa\in\N^B$, then
we write $\mu\oplus\kappa \in \N ^U$ for the valuation with
$(\mu\oplus\kappa)(x)=\mu(x)$ for $x\in A$ and
$(\mu\oplus\kappa)(x)=\kappa(x)$ for $x\in B$.

\subsection*{Computing powers.} A key observation in our proof is that
in order to compute the group element $\tau(g^m)(h)$ (in the cursor intuition, this
is the element labelling the point $h \in H$ in the wreath product element $g^m$)
 for $h \in H$ and $g\in G\wr H$, where $\sigma(g)$
has infinite order, one only has to perform at most $|\supp(g)|$ 
many multiplications in $G$, yielding a bound independent of $m$. Let us
make this precise.  Suppose $h\in H$ has infinite order.  For $h', h''\in
H$, we write $h'\preccurlyeq_h h''$ if there is an $n\ge 0$ with
$h'=h^n h''$.  Then, $\preccurlyeq_h$ is transitive. Moreover, since
$h$ has infinite order, $\preccurlyeq_h$ is also anti-symmetric and
thus a partial order. Observe that if knapsack is
decidable for $H$, given $h,h',h''\in H$, we can decide whether $h$ has
infinite order and whether $h'\preccurlyeq_h h''$.  It turns out that for
$g\in G\wr H$, the order $\preccurlyeq_{\sigma(g)}$ tells us how to
evaluate the mapping $\tau(g^m)$ at a certain element of $H$. 
Before we make this precise, we need some notation.

We will sometimes want to multiply all elements $a_i$ for $i\in I$
such that the order in which we  multiply is specified by some linear
order on $I$. If $(I,\le)$ is a finite linearly ordered set with
$I=\{i_1,\ldots,i_n\}$, $i_1<i_2<\ldots<i_n$, then we write
$\orderprod[\le]{i\in I} a_i$ for $\prod_{j=1}^n a_{i_j}$. If the
order $\le$ is clear from the context, we just write $\orderprod{i\in
  I} a_i$.
\begin{lem}\label{compute-powers-inf}
  Let $g\in G\wr H$ such that $\ord(\sigma(g))=\infty$ and let $h\in H$, $m \in \mathbb{N}$.
  Moreover let $F=\supp(g)\cap \{\sigma(g)^{-i}h\mid i\in [0,m-1]\}$.
  Then $F$ is linearly ordered by $\preccurlyeq_{\sigma(g)}$ and
\[ \tau(g^m)(h)=\orderprod[\preccurlyeq_{\sigma(g)}]{h'\in F} \tau(g)(h'). \]
\end{lem}
\begin{proof}
  By definition of $G\wr H$, we have
  $\tau(g_1g_2)(h)=\tau(g_1)(h)\cdot\tau(g_2)(\sigma(g_1)^{-1}h)$.  By
  induction, this implies
  \[ \tau(g^m)(h)=\prod_{i=0}^{m-1} \tau(g)(\sigma(g)^{-i}h)=\prod_{j=1}^n \tau(g)(\sigma(g)^{-i_j}h), \]
  where $\{i_1,\ldots,i_n\}=\{i\in[0,m-1] \mid \sigma(g)^{-i}h\in
  \supp(g)\}$ with $i_1<\cdots<i_n$.  Note that then
  $F=\{\sigma(g)^{-i_j}h \mid j\in[1,n]\}$. Since
  $\sigma(g)^{-i_j}h =\sigma(g)^{i_{j+1}-i_j} \sigma(g)^{-i_{j+1}}h$,
  we have $\sigma(g)^{-i_1}h\preccurlyeq_{\sigma(g)}\cdots
  \preccurlyeq_{\sigma(g)}\sigma(g)^{-i_n}h$.
\end{proof}

\begin{lem}\label{compute-powers-torsion}
Let $g\in G\wr H$ with $\sigma(g)=1$ and $h\in H$. Then $\tau(g^m)(h)=(\tau(g)(h))^m$.
\end{lem}
\begin{proof}
  Recall that for $g_1,g_2\in G\wr H$, we have
  $\tau(g_1g_2)(f)=\tau(g_1)(h)\cdot \tau(g_2)(\sigma(g_1)^{-1}h)$.
  Therefore, if $\sigma(g)=1$, then $\tau(g^m)(h)=\prod_{i=0}^{m-1}
  \tau(g)(\sigma(g)^{-i}h)=(\tau(g)(h))^m$.
\end{proof}

\subsection*{Addresses.}
 A central concept in our proof is that of an
address. Intuitively, a solution to the equation $E=1$ can be thought
of as a sequence of instructions on how to walk through the Cayley
graph of $H$ and place elements of $G$ at those nodes.  Here, being a
solution means that in the end, all the nodes contain the identity of
$G$.  In order to express that every node carries $1$ in the end, we
want to talk about at which points in the product
$E=g^{x_1}_1g_2^{x_2} \cdots g^{x_k}_k g_{k+1}$ a particular node is visited.  An
address is a datum that contains just enough information about such a
point to determine which element of $G$ has been placed during that
visit.

A pair $(i,h)$ with $i\in[1,k+1]$, and $h\in H$ is called an \emph{address} if
$h \in \support(g_i)$.
The set of addresses of the expression $E$ is denoted by $A$. Note
that $A$ is finite and computable. 
To each address $(i,h)$, we associate the group element $\gamma(i,h) = g_i$ of the expression $E$.

\subsection*{A linear order on addresses.} We will see that if a node
is visited more than once, then (i) each time\footnote{Here, we count
  two visits inside the same factor $g_i$, $i \in [1,k]$, with
  $\sigma(g_i)=1$ as one visit.} it does so at a different address and
(ii) the order of these visits only depends on the addresses.  To
capture the order of these visits, we define a linear order on
addresses.

We partition $A=\bigcup_{i\in[1,k+1]} A_i$, where $A_i =\{(i,h) \mid h \in \support(g_i) \}$
for $i\in [1,k+1]$. Then, for $a\in A_i$ and $a'\in
A_j$, we let $a<a'$ if and only if $i < j$.
It remains to order addresses within each $A_i$. Within $A_{k+1}$, we
pick an arbitrary order. If $i \in [1,k]$ and $\sigma(g_i)=1$, we also order $A_i$
arbitrarily. Finally, if $i \in [1,k]$ and $\sigma(g_i)$ has infinite order, then we
pick a linear order $\le$ on $A_i$ so that for
$h,h'\in\support(g_i)$, $h\preccurlyeq_{\sigma(g_i)} h'$
implies $(i,h)\le (i,h')$. Note that this is possible since
$\preccurlyeq_{\sigma(g_i)}$ is a partial order on $H$.

\subsection*{Cancelling profiles.} In order to express that a solution
for $E$ yields the identity at every node of the Cayley graph of $H$,
we need to compute the element of $G$ that is placed after the various
visits at a particular node. We therefore, associate to each address
an expression over $G$ that yields the element placed during a visit
at this address $a\in A$.  In analogy to $\tau(g)$ for $g\in G\wr H$,
we denote this expression by $\tau(a)$.  If $a=(k+1,h)$, then we set
$\tau(a)=\tau(g_{k+1})(h)$. Now, let $a = (i,h)$ for 
$i \in [1,k]$. If  $\sigma(g_i)=1$, then
$\tau(a)=\tau(g_i)(h)^{x_i}$. Finally, if 
$\sigma(g_i)$ has infinite order, then $\tau(a)=\tau(g_i)(h)$.

This allows us to express the element of $G$ that is placed at a node
$h\in H$ if $h$ has been visited with a particular set of addresses.
To each subset $C\subseteq A$, we assign the expression
$E_C=\prod_{a\in C} \tau(a)$, where the order of multiplication is
given by the linear order on $A$. Observe that only variables in
$S\subseteq\{x_1,\ldots,x_k\}$ occur in $E_C$. Therefore, given
$\kappa\in\N^S$, we can evaluate $\kappa(E_C)\in G$. We say that
$C\subseteq A$ is \emph{$\kappa$-cancelling} if $\kappa(E_C)=1$.

In order to record which sets of addresses can cancel simultaneously
(meaning: for the same valuation), we use profiles.  A \emph{profile}
is a subset of $\Powerset{A}$ (the power set of $A$). A profile $P\subseteq\Powerset{A}$ is
said to be \emph{$\kappa$-cancelling} if every $C\in P$ is
$\kappa$-cancelling. A profile is \emph{cancelling} if it is
$\kappa$-cancelling for some $\kappa\in \N^S$.

\subsection*{Clusters.} We also need to express that there is a node
$h\in H$ that is visited with a particular set of addresses. To this
end, we associate to each address $a\in A$ another expression
$\sigma(a)$.  As opposed to $\tau(a)$, the expression $\sigma(a)$ is
over $H$ and variables $M' = M\cup \{y_i \mid x_i\in
M\}$. Let $a=(i,h)\in A$. When we define $\sigma(a)$, we will also
include factors $\sigma(g_j)^{x_j}$ and $\sigma(g_j)^{y_j}$ where
$\sigma(g_j)=1$. However, since these factors do not affect the
evaluation of the expression, this should be interpreted as leaving
out such factors.
\begin{enumerate}
\item If $i = k+1$ then $\sigma(a)= \sigma(g_1)^{x_1}\cdots \sigma(g_k)^{x_k} h$.
\item If $i \in [1,k]$ then 
  $\sigma(a)= \sigma(g_1)^{x_1} \cdots \sigma(g_{i-1})^{x_{i-1}} \sigma(g_i)^{y_i}h$.
\end{enumerate}
We now want to express that when multiplying
$g_1^{\nu(x_1)}\cdots g_k^{\nu(x_k)}g_{k+1}$, there is a node $h\in
H$ such that the set of addresses with which one visits $h$ is
precisely $C\subseteq A$. In this case, we will call $C$ a cluster.

Let $\mu\in\N^M$ 
and $\mu'\in\N^{M'}$. We write $\mu'\sqsubset\mu$ if
$\mu'(x_i)=\mu(x_i)$ for $x_i\in M$ and $\mu'(y_i)\in[0,\mu(x_i)-1]$ for
every $y_i\in M'$. We can now define the set of addresses at which one
visits $h\in H$: For $h\in H$, let
\[ A_{\mu,h}=\{a\in A \mid \text{$\mu'(\sigma(a))=h$ for some $\mu'\in\N^{M'}$ with $\mu'\sqsubset\mu$}\}. \]
A subset $C\subseteq A$ is called a \emph{$\mu$-cluster} if
$C\ne\emptyset$ and there is an $h\in H$ such that $C=A_{\mu,h}$.

\begin{lem}\label{separation}
  Let $\nu\in\N^{V_E}$ with $\nu=\mu\oplus\kappa$ for $\mu\in\N^M$ and
  $\kappa\in\N^S$. Moreover, let $h\in H$ and $C=A_{\mu,h}$.
  Then $\tau(\nu(E))(h)=\kappa(E_C)$. 
\end{lem}
\begin{proof}
  
  Recall that for $k_1,k_2\in G\wr H$ and $h\in H$, we have
  $\tau(k_1k_2)(h)=\tau(k_1)(h)\cdot
  \tau(\sigma(k_1)^{-1}h)$. Therefore, we can calculate
  $\tau(\nu(E))(h)$ as
  \[ \tau(\nu(E))(h)= \prod_{i=1}^k \tau\left(g_i^{\nu(x_i)}\right) \left(\sigma(p_{i-1})^{-1}h\right) \cdot \tau(g_{k+1})\left(\sigma(p_{k})^{-1}h\right)  , \]
  where $p_i=g_1^{\nu(x_1)}\cdots g_{i}^{\nu(x_{i})}$ for $i \in [0,k]$.
  On the other hand, by definition of the linear order on $A$, we have
  \begin{multline*}
    \kappa(E_C)=\orderprod{a\in C}\kappa(\tau(a))=\left(\orderprod{a\in C\cap A_{1}} \kappa(\tau(a))\right)
    \cdots  \left(\orderprod{a\in C\cap A_{k}} \kappa(\tau(a))\right) \left(\orderprod{a\in C\cap A_{k+1}} \kappa(\tau(a))\right) .
  \end{multline*}
  Therefore, it suffices to show that 
  \begin{align}
    \tau\left(g_i^{\nu(x_i)}\right)\left(\sigma(p_{i-1})^{-1}h\right) &= \orderprod{a\in C\cap A_{i}} \kappa(\tau(a)) \label{evaluate:valuation:g} \\
    \tau(g_{k+1})\left(\sigma(p_{k})^{-1}h\right) &= \orderprod{a\in C\cap A_{k+1}} \kappa(\tau(a)), \label{evaluate:valuation:h}
  \end{align}
  for $i\in[1,k]$. 

  We begin with \cref{evaluate:valuation:h}. Note that by definition
  of $C=A_{\mu,h}$, if $a\in C\cap A_{k+1}=A_{\mu,h}\cap A_{k+1}$ with
  $a=(k+1,t)$, then there is a $\mu'\in\N^{M'}$ with $\mu'\sqsubset
  \mu$ such that $\mu'(\sigma(a))=h$. Moreover, since $a\in A_{k+1}$,
  $\sigma(a)$ contains only variables in $M$ and thus
  $\mu'(\sigma(a))=\mu(\sigma(a))=\nu(\sigma(a))$.  Note that then
  \begin{align*}
    h=\mu'(\sigma(a))=\nu(\sigma(a))=\nu(\sigma(g_1)^{x_1}\cdots \sigma(g_k)^{x_k} t)
    =\sigma(p_{k})t,
  \end{align*}
  meaning that there is only one such $t$, namely
  $t=\sigma(p_{k})^{-1}h$. 
  Moreover, recall that if
  $a=(k+1,t)$, then $\tau(a) = \tau(g_{k+1})(t)\in G$.  Therefore, the right-hand side of
  \cref{evaluate:valuation:h} is
  \[ \kappa(\tau(a))=\tau(g_{k+1})(t)=\tau(g_{k+1})\left(\sigma(p_{k})^{-1}h\right), \]
  which is the left-hand side of \cref{evaluate:valuation:h}.
  
  It remains to verify \cref{evaluate:valuation:g}.  Let us analyze
  the addresses in $C\cap A_{i}$ for $i \in [1,k]$. Consider $a\in C\cap
  A_{i}=A_{\mu,h}\cap A_{i}$ with $a=(i,t)$. Since $a\in
  A_{\mu,h}$, there is a $\mu'\sqsubset \mu$ with
  $\mu'(\sigma(a))=h$. Since $i \in [1,k]$
  we have
\begin{multline}
  h=\mu'(\sigma(a))= \mu'(\sigma(g_1)^{x_1} \cdots \sigma(g_{i-1})^{x_{i-1}} \sigma(g_i)^{y_i}t) \\
  = \sigma(g_1)^{\nu(x_1)} \cdots \sigma(g_{i-1})^{\nu(x_{i-1})} \sigma(g_i)^{\mu'(y_i)}t
  =\sigma(p_{i-1})\sigma(g_i)^{\mu'(y_i)}t.\label{evaluate:valuation:cluster}
\end{multline}
Here again, if $\sigma(g_j)=1$, we mean that the factor
$\sigma(g_j)^{\nu(x_j)}$ (resp., $\sigma(g_i)^{\mu'(y_i)}$) does not
appear. We now distinguish two cases.

\medskip
\noindent
{\em Case 1.} $\sigma(g_i)=1$. In this case,
    \cref{evaluate:valuation:cluster} tells us that  $h=\sigma(p_{i-1})t$, i.e., 
    $t=\sigma(p_{i-1})^{-1}h$. Thus, $C\cap
    A_{i}=\{(i,\sigma(p_{i-1})^{-1}h)\}$. Moreover, since
    $\sigma(g_i)=1$, $\tau(a)$ is defined as
    $(\tau(g_i)(t))^{x_i}$. Therefore, the right-hand side of
    \cref{evaluate:valuation:g} reads
    \[ (\tau(g_i)(t))^{\kappa(x_i)}= (\tau(g_i)(t))^{\nu(x_i)}=(\tau(g_i^{\nu(x_i)}))(t)=\tau\left(g_i^{\nu(x_i)}\right)\left(\sigma(p_{i-1})^{-1}h\right), \]
    where the second equality is due to
    \cref{compute-powers-torsion}. This is precisely the left-hand
    side of \cref{evaluate:valuation:g}.

\medskip
\noindent
{\em Case 2.} 
$\sigma(g_i)$ has infinite
    order. Let 
    \[ F= \supp(g_i) \cap \{\sigma(g_i)^{-j}\sigma(p_{i-1})^{-1}h \mid j\in[0,\nu(x_i)-1]\} .\]
    We claim that $t \in F$ if and only if $(i,t) \in C$.
    If $(i,t) \in C$  then 
    \Cref{evaluate:valuation:cluster} directly implies that $t \in F$. Conversely, assume that
    $t\in F$ and let
    $t=\sigma(g_i)^{-j}\sigma(p_{i-1})^{-1}h$ for $j\in[0,\nu(x_i)-1]$.
    Then, according to \cref{evaluate:valuation:cluster}, setting
    $\mu'(y_i):=j$ guarantees $\mu'(\sigma(a))=h$ for $a = (i,t)$, i.e., $(i,t)\in C$. 

    Observe that $F$ is
    linearly ordered by $\preccurlyeq_{\sigma(g_i)}$: If $j<j'$, then
    \[ \sigma(g_i)^{-j}\sigma(p_{i-1})^{-1}h =\sigma(g_i)^{j'-j}  \sigma(g_i)^{-j'}\sigma(p_{i-1})^{-1}h . \]
    Therefore, we can compute the right-hand side of \cref{evaluate:valuation:g} as
    \[ \orderprod{a\in C\cap A_i} \kappa(\tau(a))=\orderprod{a\in C\cap A_i} \tau(a)=\orderprod[\preccurlyeq_{\sigma(g_i)}]{t\in F} \tau(g_i)(t). \] 
    According to \cref{compute-powers-inf}, this equals the left-hand
    side of \cref{evaluate:valuation:g}.
\end{proof}

\begin{prop}\label{characterization-solutions}
  Let $\nu\in\N^{V_E}$ with $\nu=\mu\oplus\kappa$ for $\mu\in\N^M$ and
  $\kappa\in\N^S$.  Then $\nu(E)=1$ if and only if
  $\sigma(\nu(E))=1$ and there is a $\kappa$-cancelling profile $P$
  such that every $\mu$-cluster is contained in $P$.
\end{prop}
\begin{proof}
  Note that $\nu(E)=1$ if and only if $\tau(\nu(E))=1$ and
  $\sigma(\nu(E))=1$.  Therefore, we show that $\tau(\nu(E))=1$ if and
  only if there is a $\kappa$-cancelling profile $P$ such that every
  $\mu$-cluster is contained in $P$.

  First, let suppose that there is a $\kappa$-cancelling profile $P$
  such that every $\mu$-cluster is contained in $P$. We need to show
  that then $\tau(\nu(E))=1$, meaning $\tau(\nu(E))(h)=1$ for every
  $h\in H$. Consider the set $C=A_{\mu,h}$. If $C=\emptyset$, then by
  definition, we have $E_C = 1$. Thus, $\kappa(E_C)=1$,
  which by  \cref{separation}
  implies $\tau(\nu(E))(h)=1$.  If $C\ne\emptyset$, then $C$ is a
  $\mu$-cluster and hence $\kappa$-cancelling. Therefore, by
  \cref{separation}, $\tau(\nu(E))(h)=\kappa(E_C)=1$.  This shows that
  $\tau(\nu(E))=1$.

  Now suppose $\tau(\nu(E))=1$ and let $P\subseteq\Powerset{A}$ be the
  profile consisting of all sets $A_{\mu,h}$ with $h\in H$. Then $P$
  is $\kappa$-cancelling, because if $C\in P$ with $C=A_{\mu,h}$, then
  by \cref{separation}, we have $\kappa(E_C)=\tau(\nu(E))(h)=1$.
\end{proof}

\begin{lem}\label{cancelling-profiles}
  Suppose $\Knapsack(G^\ast)$ is decidable.  Given an instance of
  knapsack for $G\wr H$, we can compute the set of cancelling
  profiles.  If $G$ is knapsack-semilinear, then for each profile $P$,
  the set of $\kappa$ such that $P$ is $\kappa$-cancelling is
  semilinear.
\end{lem}
\begin{proof}
  A profile $P\subseteq\Powerset{A}$ is $\kappa$-cancelling if and
  only if $\kappa(E_C)=1$ for every $C\in P$. Together, the
  expressions $E_C$ for $C\in P$ constitute an instance of $\ExpEq(G)$
  (and according to \cref{kppowers-vs-expeq}, $\ExpEq(G)$ is decidable if
  $\Knapsack(G^\ast)$ is decidable) and this instance is solvable if
  and only if $P$ is cancelling. This proves the first statement of
  the \lcnamecref{cancelling-profiles}.  The second statement holds
  because the set of $\kappa\in\N^S$ such that $P$ is
  $\kappa$-cancelling is precisely $\bigcap_{C\in P} \Sol(E_C)$ and
  because the class of semilinear sets is closed under Boolean
  operations.
\end{proof}
Let $L_P\subseteq\N^M$ be the set of all $\mu\in\N^M$ such that every
$\mu$-cluster belongs to $P$.
\begin{lem}\label{solutions-presburger}
  Let $H$ be knapsack-semilinear. For every profile
  $P\subseteq\Powerset{A}$, the set $L_P$ is effectively semilinear.
\end{lem}
\begin{proof}
  We claim that the fact that every $\mu$-cluster
  belongs to $P$ can be expressed in Presburger arithmetic.
  This implies the \lcnamecref{solutions-presburger}.
    
  In addition to the variables in $M'$, we will use the variables in
  $\overline{M'}=\{\overline{x} \mid x\in M'\}$. For a knapsack
  expression $F=r_0s_1^{z_1}r_1\cdots s_m^{z_m}r_m$ with variables in
  $M'$, let $F^{-1}=r_m^{-1}(s_m^{-1})^{z_m}\cdots
  r_1^{-1}(s_1^{-1})^{z_1}r_0^{-1}$. Moreover, let
  $\overline{F}=r_0s_1^{\overline{z}_1}r_1\cdots
  s_m^{\overline{z}_m}r_m$. For $\mu\in \N^{M'}$, the valuation
  $\overline{\mu}\in\N^{\overline{M'}}$ is defined as
  $\overline{\mu}(\overline{x})=\mu(x)$ for all $x\in
  M'$. Furthermore, for $\mu\in\N^{\overline{M'}}$, we define the
  valuation $\overline{\mu}\in\N^{M'}$ by
  $\overline{\mu}(x)=\mu(\overline{x})$ for $x\in M'$. Thus if
  $\mu\in\N^{M'}$ or $\mu\in\N^{\overline{M'}}$, then
  $\overline{\overline{\mu}}=\mu$.

  As a first step, for each pair $a,b\in A$, we construct a Presburger
  formula $\eta_{a,b}$ with free variables $M'\cup \overline{M'}$ such
  that for $\mu_a\in\N^{M'}$ and $\mu_b\in\N^{\overline{M'}}$, we have
  $\mu_a\oplus\mu_b \models \eta_{a,b}$ if and only if
  $\mu_a(\sigma(a))=\overline{\mu}_b(\sigma(b))$. This is possible
  because $\mu_a(\sigma(a))=\overline{\mu}_b(\sigma(b))$ is equivalent
  to $(\mu_a\oplus\mu_b)(\sigma(a)\overline{\sigma(b)^{-1}})=1$ and the
  solution set of the knapsack expression
  $\sigma(a)\overline{\sigma(b)^{-1}}$ is effectively semilinear by
  assumption.
  
  Next, for each non-empty subset $C\subseteq A$, we construct a
  formula $\gamma_C$ with free variables in $M'$ such that
  $\mu\models\gamma_C$ if and only if $C$ is a $\mu$-cluster. Since
  $C\ne\emptyset$, we can pick a fixed $a\in C$ and let $\gamma_C$
  express the following:
  \begin{equation}
  \begin{aligned}
    \exists \mu'\in\N^{M'} \colon \mu'\sqsubset \mu&~\wedge~ \bigwedge_{b\in C} \left(\exists \mu''\in\N^{\overline{M'}} \colon \overline{\mu''}\sqsubset\mu ~\wedge~ \mu'(\sigma(a))=\overline{\mu''}(\sigma(b))\right) \\ 
    &~\wedge~ \bigwedge_{b\in A\setminus C} \left(\forall \mu''\in\N^{\overline{M'}}\colon \overline{\mu''}\sqsubset\mu ~\to~ \neg \left(\mu'(\sigma(a))=\overline{\mu''}(\sigma(b))\right)\right). \label{presburger-gamma}
  \end{aligned}
  \end{equation}
  Observe that $\mu'\sqsubset \mu$ and $\overline{\mu''}\sqsubset\mu$ are easily expressible in
  Presburger arithmetic. 
  
  Let us show that in fact $\mu\models\gamma_C$ if and only if $C$ is
  a $\mu$-cluster. Consider some $C\subseteq A$ and let $a\in C$ be
  the element picked to define $\gamma_C$.  If $\mu\models\gamma_C$,
  then there is a $\mu'\in\N^{M'}$ with the properties stated in
  \cref{presburger-gamma}. We claim that with $h:=\mu'(\sigma(a))$, we
  have $C=A_{\mu,h}$.  The second of the three conjuncts in
  \cref{presburger-gamma} states that for every $b\in C$ there is a
  $\mu''\in\N^{M'}$ such that $\mu''\sqsubset\mu$ and
  $\mu''(\sigma(b))=\mu'(\sigma(a))=h$.  Thus, $b\in A_{\mu,h}$,
  proving $C\subseteq A_{\mu,h}$.  The third conjunct states that the
  opposite is true for every $b\in A\setminus C$, so that $b\notin
  A_{\mu,h}$ for all $b\in A\setminus C$. In other words, we have
  $A_{\mu,h}\subseteq C$ and thus $A_{\mu,h}=C$.

  Conversely, suppose $C\ne\emptyset$ and $C=A_{\mu,h}$. Let $a\in C$
  be the element chosen to define $\gamma_C$.  Since $a\in A_{\mu,h}$,
  there is a $\mu'\sqsubset \mu$ with $h=\mu'(\sigma(a))$. Moreover,
  for every $b\in C$, there is a $\mu''\sqsubset\mu$ with
  $\mu''(\sigma(b))=h=\mu'(\sigma(a))$. Hence, the second conjunct is
  satsfied.  Furthermore, for every $b\in A\setminus A_{\mu,h}$, there
  is no $\mu''\sqsubset\mu$ with $\mu''(\sigma(b))=h$, meaning that
  the third conjunct is satisfied as well. Hence, $C=A_{\mu,h}$ and
  thus we have $\mu\models\gamma_C$ if and only if $C$ is a
  $\mu$-cluster.
  
  Finally, we get a formula with free variables $M$ that expresses
  that every $\mu$-cluster belongs to $P$ by writing
  $\bigwedge_{C\in\Powerset{A}\setminus P,~C\ne\emptyset} \neg\gamma_C$.
\end{proof}
We are now ready to prove \cref{thm-semilinear1,thm-semilinear2}. Let
$H$ be knapsack-semilinear and let $\Knapsack(G^\ast)$ be decidable. For
each profile $P\subseteq\Powerset{A}$, let $K_P\subseteq\N^S$ be the
set of all $\kappa\in\N^S$ such that $P$ is $\kappa$-cancelling.

Observe that for $\nu=\mu\oplus\kappa$, where $\mu \in\N^M$ and 
$\kappa\in\N^S$, the value of $\sigma(\nu(E))$
only depends on $\mu$. Moreover, the set $T\subseteq\N^M$ of all $\mu$
such that $\sigma(\nu(E))=1$ is effectively semilinear because $H$ is
knapsack-semilinear.  \Cref{characterization-solutions} tells us that
$\Sol(E)=\bigcup_{P\subseteq\Powerset{A}} K_P\oplus (L_P\cap T)$ and
\cref{solutions-presburger} states that $L_P$ is effectively
semilinear. This implies \cref{thm-semilinear1}: We can decide
solvability of $E$ by checking, for each of the finitely many profiles
$P$, whether $K_P\ne\emptyset$ (which is decidable by
\cref{cancelling-profiles}) and whether $L_P\cap
T\ne\emptyset$. Moreover, if $G$ is knapsack-semilinear, then
\cref{cancelling-profiles} tells us that $K_P$ and thus $\Sol(E)$ is
semilinear as well. This proves \cref{thm-semilinear2}.

\section{Complexity: Proof of \cref{thm-NP}} \label{complexity}

Throughout the section we fix a finitely generated group $G$.
The goal of this section is to show that if $G$ is abelian and non-trivial, then $\Knapsack(G \wr \Z)$ is 
$\NP$-complete.

\subsection{Periodic words over groups} \label{sec-periodic}

In this section we define a countable subgroup of $G^\omega$ (the direct product of $\aleph_0$ many
copies of $G$) that consists of all  periodic sequences over $G$. We show that the membership problem
for certain subgroups of this group can be solved in polynomial time if $G$ is abelian.
We believe that this is a result of independent interest which might have other applications. Therefore, we prove the best possible
complexity bound, which is $\TC^0$.\footnote{Alternatively, the reader can always replace $\TC^0$ by polynomial time in the further
arguments.}
This is the class of all problems that can be solved with uniform 
threshold circuits of polynomial size and constant depth. Here, uniformity means DLOGTIME-uniformity,
see e.g. \cite{HeAlBa02} for more details. Complete problems for $\TC^0$ are multiplication and division of binary 
encoded integers (or, more precisely, the question whether a certain bit in the output number is 1) \cite{HeAlBa02}.
$\TC^0$-complete problems in the context of group theory are 
the word problem for any infinite finitely generated solvable linear group \cite{KoeLo17}, the subgroup membership 
problem for finitely generated nilpotent groups \cite{MyasnikovW17}, the conjugacy problem 
for free solvable groups and wreath products of abelian groups \cite{MiasnikovVW17}, and the knapsack problem
for finitely generated abelian groups \cite{LohreyZ17}.

With $G^+$ we denote the set of all tuples $(g_0,\ldots, g_{q-1})$ over $G$ of arbitrary length $q \geq 1$.
With $G^\omega$ we denote the set of all mappings $f : \mathbb{N} \to G$. Elements of $G^\omega$ 
can be seen as infinite sequences (or words) over the set $G$.
We define the binary operation $\circ$ on $G^\omega$  by pointwise 
multiplication: $(f \circ g)(n) = f(n) g(n)$.
In fact, $G^\omega$ together with the multiplication $\circ$ is the direct product of $\aleph_0$ many copies
of $G$. The identity element is the mapping $\id$ with $\id(n)=1$ for all $n \in \mathbb{N}$.
For $f_1, f_2, \ldots, f_n \in G^\omega$ we write $\bigcirc_{i=1}^n f_i$ for $f_1 \circ f_2 \circ \cdots \circ f_n$.
If $G$ is abelian, we write  $\sum_{i=1}^n f_i$ for  $\bigcirc_{i=1}^n f_i$.
A function $f \in G^\omega$ is \emph{periodic with period $q \geq 1$} if $f(k) = f(k+q)$ for all $k \geq 0$.
Note that in this situation, $f$ might also be periodic with a smaller period $q' < q$.
Of course, a periodic function $f$ with period $q$ can be specified by the tuple $(f(0), \ldots, f(q-1))$.
Vice versa, a tuple $u = (g_0, \ldots, g_{q-1}) \in G^+$ defines the periodic function $f_u \in G^\omega$ with
\[ f_u(n \cdot q + r) = g_{r} \text{ for $n \geq 0$ and $0 \leq r < q$}.\]
One can view this mapping as the sequence $u^\omega$ obtained by taking infinitely many repetitions of $u$.
Let $G^\rho$ be the set of all periodic functions from $G^\omega$.
If $f_1$ is periodic with period $q_1$ and $f_2$ is periodic with period $q_2$, then 
$f_1 \circ f_2$ is periodic with period $q_1 q_2$ (in fact, $\mathrm{lcm}(q_1, q_2)$). Hence, $G^\rho$
forms a countable subgroup of $G^\omega$. Note that  $G^\rho$ is not finitely generated: The subgroup
generated by elements $f_i \in G^\rho$ with period $q_i$ ($1 \leq i \leq n$) contains only functions
with period $\mathrm{lcm}(q_1, \ldots, q_n)$.
Nevertheless, using the representation of periodic functions by elements of $G^+$ we can define
the word problem for  $G^\rho$,  \WP$(G^\rho)$ for short:
\begin{description}
\item[Input] Tuples $u_1, \ldots, u_n \in G^+$ (elements of $G$ are represented by finite words over $\Sigma$).
\item[Question] Does $\bigcirc_{i=1}^n f_{u_i}  = \id$ hold?
\end{description}
For $n \geq 0$ we define the subgroup $G^\rho_n$ of all $f \in G^\rho$ with $f(k) = 1$ for all $0 \leq k \leq n-1$.
We also consider the uniform membership problem for subgroups $G^\rho_n$, $\MEM(G^\rho_{\ast})$ for short:
\begin{description}
\item[Input] Tuples $u_1, \ldots, u_n \in G^+$ (elements of $G$ are represented by finite words over $\Sigma$) and a binary encoded number $m$.
\item[Question] Does $\bigcirc_{i=1}^n f_{u_i}$ belong to $G^\rho_m$? 
\end{description}

\begin{lem}
\WP$(G^\rho)$ is $\TC^0$-reducible to $\MEM(G^\rho_{\ast})$
\end{lem}

\begin{proof}
Let $u_1, \ldots, u_n \in G^+$ and let $q_i$ be the length of $u_i$. Let $m = \mathrm{lcm}(q_1, \ldots, q_n)$.
We have $\bigcirc_{i=1}^n f_{u_i}  = \id$ if and only $\bigcirc_{i=1}^n f_{u_i}$ belongs to $G^\rho_m$.
\end{proof}

\begin{thm} \label{thm-abelian-membership}
For every finitely generated abelian group $G$, $\MEM(G^\rho_{\ast})$ belongs to $\TC^0$.
\end{thm}

\begin{proof} 
Since the word problem for a finitely generated abelian group belongs to $\TC^0$, it suffices to show the following
claim:

\medskip
\noindent
{\em Claim:}  Let $u_1, \ldots, u_n \in G^+$ and let $q_i$ be the length of $u_i$.
Let $f = \sum_{i=1}^n f_{u_i}$.
If there exists a position $m$ such that $f(m) \neq 0$, then there exists a position $m < \sum_{i=1}^n q_i$ 
such that $f(m) \neq 0$.  

\medskip
\noindent
Let $m \geq  \sum_{i=1}^n q_i$. We show that if $f(j)=0$ for all $j$ with $m-\sum_{i=1}^n q_i \leq j < m$, then also $f(m)=0$, 
which proves the above claim.

Hence, let us assume that $f(j)=0$ for all $j$ with $m-\sum_{i=1}^n q_i \leq j < m$.
Note that $f_{u_i}(j)=f_{u_i}(j-q_i)$ for all $j \geq q_i$ and $1 \leq i \leq n$.  
For $M \subseteq [1,n]$ let $q_M= \sum_{i \in M} q_i$. 
Moreover, for $1 \leq k \leq n$ let $\mathcal{M}_k = \{ M \subseteq [1,n], |M|=k \}$.
For all $1 \leq k \leq n-1$ we get
\begin{eqnarray*}
 \sum_{M \in \mathcal{M}_k} \sum_{i \in M} f_{u_i}(m-q_M) &=&  - \sum_{M \in \mathcal{M}_k} \sum_{i \in [1,n] \setminus M} f_{u_i}(m-q_M) \\
 &= & - \sum_{M \in \mathcal{M}_k} \sum_{i \in [1,n] \setminus M} f_{u_i}(m-q_M-q_i) \\
 &= & - \sum_{i=1}^n \sum_{M \in \mathcal{M}_k, i \notin M} f_{u_i}(m-q_{M \cup \{i\}} ) \\
 &= & -\sum_{i=1}^n \sum_{M \in \mathcal{M}_{k+1}, i \in M} f_{u_i}(m-q_M) \\
 &= & - \sum_{M \in \mathcal{M}_{k+1}} \sum_{i \in M} f_{u_i}(m-q_M).
 \end{eqnarray*}
We can write
\[ f(m)=\sum_{i=1}^n f_{u_i}(m)=\sum_{i=1}^n f_{u_i}(m- q_i) =\sum_{M \in \mathcal{M}_1} \sum_{i \in M} f_{u_i}(m-q_M).\]
From the above identities we get by  induction:
\begin{eqnarray*}
f(m) &=& (-1)^{n+1}\sum_{M \in \mathcal{M}_n} \sum_{i \in M} f_{u_i}(m-q_M) \\
&=& (-1)^{n+1} \sum_{i \in [1,n]} f_{u_i}(m-q_{[1,n]}) \\
&=& (-1)^{n+1} f(m-\sum_{i=1}^n q_i)=0.
 \end{eqnarray*}
This proves the claim and hence the theorem.
\end{proof}

\subsection{Automata for Cayley representations}

The goal of this section is to show that if $\ExpEq(G)$ and $\MEM(G^\rho_{\ast})$ both belong to $\NP$,
then also $\mathrm{KP}(G \wr \Z)$ belongs to $\NP$.

An interval $[a,b] \subseteq \Z$ {\em supports} an element $(f,d) \in G \wr \Z$
if $\{0,d\} \cup \supp(f) \subseteq [a,b]$.
If $(f,d) \in G \wr \Z$ is a product of length $n$ over the generators,
then the minimal interval $[a,b]$ which supports $(f,d)$ satisfies $b-a \le n$.
A knapsack expression $E = v_0 u_1^{x_1} v_1 \cdots u_k^{x_k} v_k$ is called {\em rigid}
if each $u_i$ evaluates to an element $(f_i,0) \in G \wr \Z$.
Intuitively, the movement of the cursor is independent from the values of the variables $x_i$ up to repetition of loops.
In particular, every variable-free expression is rigid.

In the following we define so called Cayley representations of rigid knapsack expressions.
This is a finite word, where every symbol is a marked knapsack expression over 
$G$. A marked knapsack expression over $G$ is of the form $E$, $\overline{E}$, $\underline{E}$, 
or $\overline{\underline{E}}$, where $E$ is 
a knapsack expression over $G$. We say that $\overline{E}$ and $\overline{\underline{E}}$
(resp., $\underline{E}$ and $\overline{\underline{E}}$) are top-marked (resp., bottom-marked). 

Let $E = v_0 u_1^{x_1} v_1 \cdots u_k^{x_k} v_k$ be a rigid knapsack expression over $G \wr \mathbb{Z}$.
For an assignment $\nu$ let $(f_\nu,d) \in G \wr \Z$ be the element to which $\nu(E)$ evaluates, i.e. $(f_\nu, d)=\nu(E)$.
Note that $d$ does not depend on $\nu$.
Because of the rigidity of $E$, there is an interval  $[a,b] \subseteq \Z$ that supports $(f_\nu,d)$ for all assignments $\nu$.
For each $j \in [a,b]$ let $E_j$ be a knapsack expression over $G$ with the variables $x_1, \dots, x_k$
such that $f_\nu(j) = \nu(E_j)$ for all assignments $\nu$.
Then we call the formal expression
\[
	r = \begin{cases}
	E_a \, E_{a+1} \, \cdots \, E_{-1} \, \overline{E_0} \, E_1 \, \cdots \, E_{d-1} \, \underline{E_d} \, E_{d+1} \, \cdots \, E_b & \text{ if } d > 0 \\
	E_a \, E_{a+1} \, \cdots \, E_{-1} \, \overline{\underline{E_0}} \, E_1 \, \cdots \, E_b & \text{ if }  d = 0 \\
	E_a \, E_{a+1} \, \cdots \, E_{d-1} \, \underline{E_d} \, E_{d+1} \, \cdots \, E_{-1} \, \overline{E_0} \, E_1 \, \cdots \, E_b & \text{ if } d < 0
	\end{cases}	.
\]
a {\em Cayley representation} of $E$ (or $E$ is {\em represented} by $r$).
Formally, a Cayley representation is a sequence of marked knapsack expressions.
For a Cayley representation $r$, we denote by $|r|$ the number of knapsack expressions in the sequence.
If necessary, we separate consecutive marked knapsack expressions in
$r$ by commas. For instance, if $a_1$ and $a_2$ are generators of $G$, then
$\overline{a_1}, a_2 a_1, \underline{a_2}$ is a Cayley representation
of length 3, whereas $\overline{a_1}, a_2, a_1, \underline{a_2}$ is a
Cayley representation of length 4.  By this definition, $r$ depends on
the chosen supporting interval $[a,b]$. However, compared to the
representation of the minimal supporting interval, any other Cayley
representation differs only by adding $1$'s (i.e., trivial knapsack
expressions over $G$) at the left and right end of $r$.

A Cayley representation of $E$ records for each point in $\Z$ an
expression that describes which element will be placed at that
point. Multiplying an element of $G\wr\Z$ always begins at a
particular cursor position; in a Cayley representation, the marker on
top specifies the expression that is placed at the cursor position in
the beginning. Moreover, a Cayley representation describes how the
cursor changes when multiplying $\nu(E)$: The marker on the bottom
specifies where the cursor is located in the end.

\begin{ex} \label{ex-representation}
Let us consider the wreath product $F_2 \wr \Z$ where $F_2$ is the free group generated by $\{a,b\}$
and $\Z$ is generated by $t$.
Consider the rigid knapsack expression $E = u_1^x u_{2}^{} u_3^y u_4^5$ where
\begin{itemize}
	\item $u_1 = a t^{-1} a t^2 b t^{-1}$, represented by $a \, \overline{\underline{a}} \, b$,
	\item $u_2 = t$, represented by $\overline{1} \, \underline{1}$,
	\item $u_3 = b t b t b t^{-2}$, represented by $\overline{\underline{b}} \, b \, b$,
	\item 
	$u_4 = a t^{-1} b t^2 b^{-1} t a t a t^{-1}$,
	represented by $b \, \overline{a} \, b^{-1} \, \underline{a} \, a$.
\end{itemize}
A Cayley representation of $u_1^x$ is $a^x \, \overline{\underline{a^x}} \, b^{-1}$
and a Cayley representation of $u_3^y$ is $\overline{\underline{b^y}} \, b^y \, b^y$.
The diagram in \cref{fig:cayley-rep} illustrates how to compute a Cayley representation $r$ of $E$, which is shown
in the bottom line. Here, we have chosen the supporting interval minimal. Note that if we replace the exponents
$5$ in $u_4^5$ by a larger number, then we only increase the number of repetitions of the factor $a, a^2$ in the 
Cayley representation. 
\end{ex}

\begin{figure}
\begin{center}
\begin{tabular}{|c|c|c|c|c|c|c|c|c|c|c|c|c|c|c|c|}
-1 & 0  & 1  & 2  & 3 & 4 & 5 & 6 & 7 & 8 & 9 & 10 & 11 & 12  \\ \thickhline
$a^x$ & $\overline{\underline{a^x}}$ & $b^x$ &&&&&&&&&&& \\
& $\overline{1}$ & $\underline{1}$ &&&&&&&&&&& \\
& & $\overline{\underline{b^y}}$ & $b^y$ & $b^y$ &&&&&&&&& \\
& $b$ & $\overline{a}$ & $b^{-1}$ & $\underline{a}$ & $a$ &&&&&&&& \\
& & & $b$ & $\overline{a}$ & $b^{-1}$ & $\underline{a}$ & $a$ &&&&&& \\
& & & & & $b$ & $\overline{a}$ & $b^{-1}$ & $\underline{a}$ & $a$ &&&& \\
& & & & & & & $b$ & $\overline{a}$ & $b^{-1}$ & $\underline{a}$ & $a$ && \\
& & & & & & & & & $b$ & $\overline{a}$ & $b^{-1}$ & $\underline{a}$ & $a$ \\ \hline
$a^x$ & $\overline{a^xb}$ & $b^x b^y a$ & $b^y$ & $b^y a^2$ & $a$ & $a^2$ & $a$ & $a^2$ & $a$ & $a^2$ & $ab^{-1}$ & $\underline{a}$ & $a$
\end{tabular}
\end{center}
\caption{Cayley representation}
\label{fig:cayley-rep}
\end{figure}

\cref{ex-representation} also illustrates the concept of so called consistent tuples,
which will be used later.
A tuple $(\gamma_1, \dots, \gamma_n)$, where every $\gamma_i$ is a marked knapsack expression over $G$
is {\em consistent} if, whenever $\gamma_i$ is bottom-marked and $i<n$, then  $\gamma_{i+1}$ is top-marked.
Every column in \cref{fig:cayley-rep} is a consistent tuple.

Let $E$ be an arbitrary knapsack expression over $G \wr \Z$. We can assume that
$E$ has the form $u_1^{x_1} \cdots u_k^{x_k} u_{k+1}$.
We partition the set of variables $X = \{x_1, \dots, x_k\}$ as $X=X_0 \cup X_1$,
where $X_0$ contains all variables $x_i$ where $u_i$ evaluates to an element $(f,0) \in G \wr \Z$,
and $X_1$ contains all other variables. For a partial assignment $\nu \colon X_1 \to \N$  we obtain
a rigid knapsack expression $E_\nu$ by replacing in $E$ every variable $x_i \in X_1$ by
$\nu(x_i)$.
A set $R$ of Cayley representations is a {\em set representation} of $E$ if
\begin{itemize}
	\item for each assignment $\nu \colon X_1 \to \N$ there exists $r \in R$
	such that $r$ represents $E_\nu$,
	\item for each $r \in R$ there exists an assignment $\nu \colon X_1 \to \N$
	such that $r$ represents $E_\nu$ and $\nu(x) \le |r|$ for all $x \in X_1$.
\end{itemize}

\begin{ex} \label{ex-representation2}
Let us consider again the wreath product $F_2 \wr \Z$
and consider the (non-rigid) knapsack expression $E' = u_1^x u_{2}^{} u_3^y u_4^z$ where
$u_1, u_2, u_3, u_4$ are taken from \cref{ex-representation}. We have $X_0 = \{ x,y \}$ and $X_1 = \{ z \}$.
For $z=5$ we obtained in \cref{ex-representation} the Cayley representation 
\[
a^x, \overline{a^xb}, b^x b^y a, b^y, b^y a^2, a, a^2, a, a^2, a, a^2, ab^{-1}, \underline{a}, a .
\]
A set representation $R$ of $E'$ consists of the following Cayley representations:
\begin{itemize}
\item $a^x, \overline{a^x}, \underline{b^x b^y}, b^y, b^y$ for $\nu(z) = 0$,
\item $a^x, \overline{a^xb}, b^x b^y a, b^y b^{-1}, \underline{b^y a}, a$ for $\nu(z)=1$,
\item  $a^x, \overline{a^xb}, b^x b^y a, b^y, b^y a^2, \underbrace{a, a^2, \ldots, a, a^2}_{\text{$\nu(z)-2$ times}}, a b^{-1}, \underline{a}, a$ for $\nu(z) \geq 2$.
\end{itemize}
Only finitely many different marked knapsack expressions appear in this set representation $R$, and 
$R$ is clearly a regular language over the finite alphabet consisting of this finitely many marked knapsack expressions.
\end{ex}
In the following, we will show that for every knapsack expression $E=u_1^{x_1} \cdots u_k^{x_k} u_{k+1}$  there exists a 
non-deterministic finite automaton (NFA)
that accepts a set representation of $E$, whose size is exponential in $n = |E|$. 
First, we consider the blocks $u_1^{x_1}, \ldots, u_k^{x_k}, u_{k+1}$.

\begin{lem} \label{lemma-A_i} One can compute in polynomial time for
  each $1 \le i \le k+1$ an NFA $\mathcal{A}_i$ of size $|u_i|^{O(1)}$
  that recognizes a set representation of $u_i^{x_i}$ or $u_{k+1}$.
\end{lem}

\begin{proof}
Let us do a case distinction.

\medskip
\noindent
{\em Case 1.}
	Consider an expression $u_i^{x_i}$ where $x_i \in X_0$,
	i.e. $u_i$ evaluates to some element $(f,0) \in G \wr \Z$.
	Let $[a,b]$ be the minimal interval which supports $(f,0)$. Thus, $b-a \leq |u_i|$.
	Then
	\[
		r_i = f(a)^{x_i} \cdots f(-1)^{x_i} \overline{\underline{f(0)^{x_i}}}  f(1)^{x_i} \cdots f(b)^{x_i}
	\]
	is a Cayley representation of $u_i^{x_i}$ where $|r_i| = b-a+1 \le |u_i| + 1$.
	Clearly, $\{r_i\}$ is a set representation of $u_i^{x_i}$,
	which is recognized by an NFA $\mathcal{A}_i$ of size $|r_i|+1 \le |u_i|+2$.
	
\medskip
\noindent
{\em Case 2.}	
	Similarly, for the word $u_{k+1}$ we obtain a Cayley representation $r_{k+1}$ as above except
	that the exponents $x_i$ are not present.
	Again, $\{r_{k+1}\}$ is a set representation of $u_{k+1}$,
	which is recognized by an NFA $\mathcal{A}_{k+1}$ of size $|u_{k+1}|+2$.

\medskip
\noindent
{\em Case 3.}	
Consider an expression $u_i^{x_i}$ where $x_i \in X_1$,
	i.e., $u_i$ evaluates to some element $(f,d) \in G \wr \Z$ where $d \neq 0$.
	Let $[a,b]$ be a minimal interval which supports $(f,d)$, hence $b-a \le |u_i|$.
	
	We only consider the case $d > 0$; at the end we say how to modify the construction for $d < 0$.
	Consider the word
	\[
		r_i = f(a)  \cdots f(-1) \overline{f(0)} \cdots \underline{f(d)} f(1) \cdots  f(b),
	\]
	which is a Cayley representation of $(f,d)$.
	We will prove that there is an NFA $\mathcal{A}_i$ with $\varepsilon$-transitions of size $O(|r_i|^2) = O(|u_i|^2)$
	which recognizes a set representation of $u_i^{x_i}$. This set representation has to contain
	a Cayley representation of every $u_i^m$ (a variable-free knapsack expression over $G$) for $m \geq 0$.

	First we define an auxiliary automaton $\mathcal{B}$. \cref{ex-run} shows an example of the following 
	construction.
	Let $\Gamma$ be the alphabet of $r_i$ (a set of possibly marked elements of $G$)
	and define $g \colon [a,b] \to \Gamma$ by
	\[
		g(c) = \begin{cases}
			\overline{f(0)} & \text{if $c = 0$} \\
			\underline{f(d)} & \text{if $c = d$} \\
			f(c) & \text{otherwise.}
		\end{cases}
	\]
	The state set of $\mathcal{B}$ is the set $Q$ of all decreasing arithmetic progressions
	$(s, s - d, s - 2d, \dots, s - \ell d)$ in the interval $[a,b]$ where $\ell \ge 0$
	together with a unique final state $\top$.
	It is not hard to see that $|Q| = O(|r_i|^2)$.
	For each state $(s_0, \dots, s_\ell) \in Q$ we define the marked $G$-element
	\[
	\alpha(s_0, \dots, s_\ell) =  \begin{cases}
	  f(s_0) \cdots f(s_\ell)  & \text{if neither $g(s_0)$ is top-marked nor $g(s_\ell)$ is bottom-marked} \\[1mm]
	  \overline{f(s_0) \cdots f(s_\ell)}  & \text{if $g(s_0)$ is top-marked} \\[1mm]
	 \underline{f(s_0) \cdots f(s_\ell)}  & \text{if $g(s_\ell)$ is bottom-marked} 
	\end{cases}
	\]
	Since $d > 0$ it cannot happen that $g(s_0)$ is top-marked and at the same time $g(s_\ell)$ is bottom-marked.
	The initial state is the 1-tuple $(a)$.
	For each state $(s_0, \dots, s_\ell) \in Q$ and $\gamma = \alpha(s_0, \dots, s_\ell)$
	the automaton has the following transitions:
	\begin{itemize}
		\item $(s_0, \dots, s_\ell) \xrightarrow{\varepsilon} (s_0, \dots, s_\ell,a)$ if $s_\ell = a+d$
		\item $(s_0, \dots, s_\ell) \xrightarrow{\gamma} (s_0+1, \dots, s_\ell+1)$ if $s_0 < b$
		\item $(s_0, \dots, s_\ell) \xrightarrow{\gamma} (s_1+1, \dots, s_{\ell}+1)$ if $s_0 = b$ and $\ell \ge 1$
		\item $(s_0, \dots, s_\ell) \xrightarrow{\gamma} \top$ if $s_0 = b$ and $\ell = 0$
	\end{itemize}
	Finally we take the union with another automaton which accepts the singleton $\{ \overline{\underline{1}} \}$.
	This yields the desired automaton $\mathcal{A}_i$.
	
	If $d<0$ we can consider the group element
	$(f',-d)$ with $f' \colon [-b,-a] \to G$, $f'(c) = f(-c)$ for $-b \leq c \leq -a$. 
	We then do the above automaton construction for $(f',-d)$. From the resulting NFA
	we finally construct an automaton for the reversed language.
This proves the lemma.
\end{proof}

\begin{ex} \label{ex-run}
Below is a run of the automaton for $(a t^{-1} b t^2 b^{-1} t a t a t^{-1})^x$ on  the word
\[ b,\overline{a},1,(a^2,a)^3,a^2,ab^{-1},\underline{a},a.
\]
\cref{fig:cayley-rep2} shows how this
word is produced from $(a t^{-1} b t^2 b^{-1} t a t a t^{-1})^5$. The last line shows the tuple of relative
positions in the currently  ``active'' copies of $b,a,b^{-1},a,a$. The positions are $-1,0,1,2,3$. For instance,
the tuple $(3,1,-1)$ means that currently three copies of $b,a,b^{-1},a,a$ are active. The current position
in the first copy is 3, the current position in the second copy is 1, and the current position in the third copy is -1.
These tuples are states in the run below. The only additional states $(1)$ and $(3,1)$ in the run are origins of $\varepsilon$-transitions,
which add new copies of $b,a,b^{-1},a,a$.
\begin{align*}
& (-1) \xrightarrow{b} (0) \xrightarrow{\overline{a}} (1) \xrightarrow{\varepsilon} (1,-1) \xrightarrow{1} \\
& (2,0)  \xrightarrow{a^2} (3,1)  \xrightarrow{\varepsilon}  (3,1,-1)  \xrightarrow{a} \\
& (2,0)  \xrightarrow{a^2} (3,1)  \xrightarrow{\varepsilon}  (3,1,-1)  \xrightarrow{a} \\
& (2,0)  \xrightarrow{a^2} (3,1)  \xrightarrow{\varepsilon}  (3,1,-1)  \xrightarrow{a} \\
& (2,0)  \xrightarrow{a^2} (3,1) \xrightarrow{ab^{-1}} (2)   \xrightarrow{\underline{a}} (3)  \xrightarrow{a} \top
\end{align*}
\end{ex}
\begin{figure}
\begin{center}
\begin{tabular}{|c|c|c|c|c|c|c|c|c|c|c|c|c|c|c|c|}
 \hline
 $b$ & $\overline{a}$ & $b^{-1}$ & $\underline{a}$ & $a$ &&&&&&&& \\
 & & $b$ & $\overline{a}$ & $b^{-1}$ & $\underline{a}$ & $a$ &&&&&& \\
 & & & & $b$ & $\overline{a}$ & $b^{-1}$ & $\underline{a}$ & $a$ &&&& \\
 & & & & & & $b$ & $\overline{a}$ & $b^{-1}$ & $\underline{a}$ & $a$ && \\
 & & & & & & & & $b$ & $\overline{a}$ & $b^{-1}$ & $\underline{a}$ & $a$  \\ \hline
  $b$ & $\overline{a}$ & $1$ & $a^2$ & $a$ & $a^2$ & $a$ & $a^2$ & $a$ & $a^2$ & $ab^{-1}$ & $\underline{a}$ & $a$  \\ \hline
  (-1)  & (0)  & (1,-1) & (2,0) & (3,1,-1) & (2,0) & (3,1,-1) & (2,0) & (3,1,-1) & (2,0) &(3,1) & (2) & (3) 
\end{tabular}
\end{center}
\caption{A run of the automaton for $(a t^{-1} b t^2 b^{-1} t a t a t^{-1})^x$}
\label{fig:cayley-rep2}
\end{figure}
A language $L \subseteq \Sigma^*$ is {\em bounded} if there exist
words $\beta_1, \dots, \beta_n \in \Sigma^*$ such that $L \subseteq
\beta_1^* \cdots \beta_n^*$.  It will be convenient to use the
following characterization. For states $p,q$ of an automaton
$\mathcal{B}$, let $L_{p,q}(\mathcal{B})$ be the set of all words read
on a path from $p$ to $q$. An NFA $\mathcal{B}$ recognizes a bounded
language if and only if for every state $q$, the language
$L_{q,q}(\mathcal{B})$ is commutative, meaning that 
$uv=vu$ for any $u,v\in L_{q,q}(\mathcal{B})$~\cite{GaKrRaSh2010}.

\begin{lem}\label{construct-bounded-words}
  Given an NFA $\mathcal{B}$ that recognizes a bounded language, one
  can compute in polynomial time words $\beta_1,\ldots,\beta_n$ with
  $L(\mathcal{B})\subseteq\beta_1^*\cdots\beta_n^*$.
\end{lem}
\begin{proof}
  For any two states $p,q$ with $L_{p,q}(\mathcal{B})\ne\emptyset$, compute a
  shortest word $w_{p,q}\in L_{p,q}(\mathcal{B})$ and let $P_q=u_1^*\cdots u_m^*$,
  where $w_{q,q}=u_1\cdots u_m$ and $u_1,\ldots,u_m$ are letters.

  We first prove the \lcnamecref{construct-bounded-words} for the
  languages $L_{p,q}=L_{p,q}(\mathcal{B})$ if $p,q$ lie in the same
  strongly connected component.  Any two words in $L_{p,q}$ have to be
  comparable in the prefix order: Otherwise we could construct two
  distinct words of equal length in $L_{p,p}$, contradicting the
  commutativity of $L_{p,p}$. Since $w_{p,q}w_{q,q}^*\subseteq
  L_{p,q}$, this means that every word in $L_{p,q}$ must be a prefix
  of a word in $w_{p,q}w_{q,q}^*$. In particular, we have
  $L_{p,q}\subseteq w_{p,q}^*w_{q,q}^*P_q$.

  In the general case, we assume that $\mathcal{B}$ has only one
  initial state $s$. We decompose
  $\mathcal{B}$ into strongly connected components, yielding a
  directed acyclic graph $\Gamma$ with vertices $V$.  For $i \leq |V|$,
  let $D_i=\{v\in V \mid \text{$v$ has distance $i$ from $[s]$ in $\Gamma$}\}$,
  where $[s]$ denotes the strongly connected component of $s$. Observe
  that $L(\mathcal{B})\subseteq \prod_{i=0}^{|V|} \prod_{v\in D_i}
  \prod_{p,q\in v} L_{p,q}$, where the two innermost products are carried out
  in an arbitrary order. Since we have established the
  \lcnamecref{construct-bounded-words} in the case of the $L_{p,q}$,
  this tells us how to perform the computation for $L(\mathcal{B})$.
\end{proof}

\begin{lem}
	\label{lem:bounded}
	The NFAs $\mathcal{A}_i$ from \cref{lemma-A_i} recognize bounded languages.
\end{lem}

\begin{proof}
  The statement is clear for the automata which recognize singleton
  languages in cases 1.~and 2.  Consider the constructed automaton
  $\mathcal{B}$ from case 3.  It is almost deterministic in the
  following sense: Every state in $\mathcal{B}$ has at most one
  outgoing transition labelled by a symbol from the alphabet and at
  most one outgoing $\varepsilon$-transition.

  We partition its state set as $Q=Q_0\uplus Q_1$, where $Q_0$
  consists of those states $(s_0,\ldots,s_\ell)$ where $s_\ell\le
  a+d$. Since there is no transition from $Q_1$ to $Q_0$, every
  strongly connected component is either entirely within $Q_0$ or
  entirely within $Q_1$. If a state $q$ has an outgoing
  $\varepsilon$-transition, then $q\in Q_0$ and all
  non-$\varepsilon$-transitions from $q$ lead into $Q_1$. Therefore,
  every state in $\mathcal{B}$ has at most one outgoing transition
  that leads into the same strongly connected component.  Thus,
  every strongly connected component is a directed cycle, meaning
  that $L_{q,q}(\mathcal{B})=w^*$, where $w$ is the word read on that cycle.
  Hence, $\mathcal{B}$ recognizes a bounded language.  Hence also
  $L(\mathcal{A}_i) = L(\mathcal{B}) \cup \{ \overline{\underline{1}}
  \}$ is bounded.
\end{proof}

\begin{lem}
	There exists an NFA $\mathcal{A}$ of size $\prod_{i=1}^{k+1} O(|u_i|) \leq 2^{O(n \log n)}$ which recognizes a set representation of $E$,
	where $n = |E|$.
\end{lem}

\begin{proof}
Reconsider the automata $\mathcal{A}_i$ from \cref{lemma-A_i}.
We first ensure that for all $1 \le i \le k+1$ we have
$L(\mathcal{A}_i) = 1^* \, L(\mathcal{A}_i) \, 1^*$,
which can be achieved using two new states in $\mathcal{A}_i$.
Let $\mathcal{E}_i$ be the finite alphabet of marked knapsack expressions that occur
as labels in $\mathcal{A}_i$ and let $\mathcal{E}$ be the set of consistent tuples in the cartesian
product $\mathcal{E}_1\times\cdots\times \mathcal{E}_{k+1}$.

Let $\mathcal{A}'$ be the following product NFA over the alphabet $\mathcal{E}$.
It stores a $(k+1)$-tuple of states (one for each NFA $\mathcal{A}_i$).
On input of a consistent tuple $(\gamma_1, \dots, \gamma_{k+1}) \in \mathcal{E}$ 
it reads $\gamma_i$ into $\mathcal{A}_i$.
The size of $\mathcal{A}'$ is $\prod_{i=1}^{k+1} O(|u_i|) \leq  2^{O(n \log n)}$.
To obtain the NFA $\mathcal{A}$ we project the transition labels of $\mathcal{A}'$ as follows:
Let $(\gamma_1, \dots, \gamma_{k+1}) \in \mathcal{E}$ and let $(\chi_1, \ldots, \chi_{k+1})$ obtained
by removing all markings from the $\gamma_i$. We then replace the transition label
$(\chi_1, \dots, \chi_{k+1})$ by
\begin{itemize}
\item $\chi_1 \cdots \chi_{k+1}$ if neither $\chi_1$ is top-marked nor $\chi_{k+1}$ is bottom-marked,
\item $\overline{\chi_1 \cdots \chi_{k+1}}$ if $\chi_1$ is top-marked and $\chi_{k+1}$ is not bottom-marked,
\item $\underline{\chi_1 \cdots \chi_{k+1}}$ if $\chi_1$ is not top-marked and $\chi_{k+1}$ is bottom-marked,
\item $\underline{\overline{\chi_1 \cdots \chi_{k+1}}}$ if $\chi_1$ is top-marked and $\chi_{k+1}$ is bottom-marked.
\end{itemize}
One can verify that $\mathcal{A}$ recognizes a set representation of $E$.
\end{proof}

\begin{prop} \label{prop-NP-reduction}
	Let $G$ be a finitely generated abelian group.
	If $\ExpEq(G) \in \mathsf{NP}$ and $\MEM(G^\rho_{\ast}) \in \mathsf{NP}$,
	then also $\mathrm{KP}(G \wr \Z) \in \mathsf{NP}$.
\end{prop}

\begin{proof}
We first claim that, if $E = 1$ is solvable, then there exists a solution $\nu$
such that $\nu(x)$ is exponentially bounded in $n$ for all $x \in X_1$.
Assume that $\nu$ is a solution for $E=1$.
From the NFA $\mathcal{A}$, we obtain an automaton $\mathcal{A}'$ by replacing
each knapsack expression in the alphabet of $\mathcal{A}$ by its value unter $\nu$ in $G$.
Then, $\mathcal{A}'$ has the same number of states as $\mathcal{A}$, hence at most $2^{O(n\log n)}$.
Moreover, $\mathcal{A}'$ accepts
a Cayley representation of the identity of $G \wr \Z$ (which is just a sequence of $1$'s).
Due to the size bound, $\mathcal{A}'$ accepts such a representation of length $2^{O(n \log n)}$.
Since $\mathcal{A}$ accepts a set representation of $E$, this short computation corresponds to a solution $\nu'$. By definition of a set representation,
for each $x\in X_1$, $\mathcal{A}'$ makes at least $\nu'(x)$ steps. Therefore, $\nu'(x)$ is bounded
exponentially for $x\in X_1$.

Since each $\mathcal{A}_i$ accepts a set representation
of $u_i^{x_i}$, $i\in[1,k]$ or of $u_{k+1}$, this implies that solvability of $E$ is
witnessed by words $\alpha_1,\ldots,\alpha_{k+1}$ with
$\alpha_i\in L(\mathcal{A}_i)$ for $i\in[1,k+1]$ whose length is bounded
exponentially.

In the following we will encode exponentially long words as follows:
A {\em cycle compression} of a word $w$ is a sequence $(\beta_1, \ell_1, \dots, \beta_m, \ell_m)$
where each $\beta_i$ is a word and each $\ell_i \ge 0$ is a binary encoded integer
such that there exists a factorization $w = w_1 \cdots w_m$
and each factor $w_i$ is the prefix of $\beta_i^\omega$ of length $\ell_i$.
Each $w_i$ is called a {\em cycle factor} in $w$.

We need the following simple observation.
Let $(\beta_1, \ell_1, \dots, \beta_m, \ell_m)$ be a cycle compression
of a word $w$ with the corresponding factorization $w = w_1 \cdots w_m$.
Given a position $p$ in $w$ which yields factorizations $w = uv$, $u = w_1 \cdots w_{i-1} w_i'$,
$v = w_i'' w_{i+1} \cdots w_m$ and $w_i = w_i' w_i''$.
{\em Splitting} $(\beta_1, \ell_1, \dots, \beta_m, \ell_m)$ at position $p$
yields the unique cycle compression of $w$ of the form
\[
	(\beta_1, \ell_1, \dots, \beta_{i-1}, \ell_{i-1}, \beta_i', \ell_i', \beta_i'', \ell_i'', \dots, \beta_m, \ell_m)\
\]
where $|w_i'| = \ell_i'$ and $|w_i''| = \ell_i''$.
Clearly, splitting can be performed in polynomial time.
With the help of splitting operations we can also remove a given set of positions from a cycle compressed word in polynomial time.

This leads us to our $\mathsf{NP}$-algorithm: First we construct the
NFAs $\mathcal{A}_i$ as above.  By \cref{lem:bounded} each NFA
$\mathcal{A}_i$ recognizes a bounded language. Hence for each $i \in [1,k+1]$,
\cref{construct-bounded-words} allows us to compute in
polynomial time words $\beta_{i,1}, \dots, \beta_{i,m_i}$ such that
$L(\mathcal{A}_i) \subseteq \beta_{i,1}^* \cdots \beta_{i,m_i}^*$.  For
each $\mathcal{A}_i$ we guess a cycle compression
$(\beta_{i,1}, \ell_{i,1}, \dots, \beta_{i,m_i}, \ell_{i,m_i})$
of a word $\alpha_i$ such that the words
$\alpha_1,\ldots,\alpha_{k+1}$ have equal length $\ell$.
Then, we test in polynomial time whether
$\alpha_i$ is accepted by $\mathcal{A}_i$ (this is a restricted case
of the {\em compressed membership problem} of a regular language
\cite{Loh12survey}).  Next we verify in polynomial time whether the
markers of the $\alpha_i$ are consistent and whether the position of
the origin in $\alpha_1$ coincides with the position of the cursor in
$\alpha_{k+1}$.  If so, we remove all markers from the words
$\alpha_i$.

Finally we reduce to instances of $\ExpEq(G)$ and
$\MEM(G^\rho_{\ast})$.
Denote with $P = \{p_1, \dots, p_r\}\subseteq
[1,\ell]$ the set of positions $p$ such that there exists a
variable $x_i \in X_0$ occurring in $\alpha_i[p]$, which is the
expression at position $p$ in $\alpha_i$. Note that if a variable
$x_i\in X_0$ occurs in $\alpha_i$, then by definition of $X_0$ and set
representations, $\alpha_i$ contains at most $|u_i|^{O(1)}$ positions
with an expression $\ne 1$. We can therefore compute $P$ in polynomial
time and obtain an instance of $\ExpEq(G)$ containing the expression
$\alpha_1[p_j]\cdots \alpha_{k+1}[p_j]$
for each $j\in[1,r]$.
We then remove the positions in $P$ from the words $\alpha_i$ and compute
cycle compressions $(\beta_{i,1}, \ell_{i,1}, \dots, \beta_{i,m_i}, \ell_{i,m_i})$
of the new words $\alpha_i$ in polynomial time.

The remaining words reduce to instances of $\MEM(G^\rho_{\ast})$ as follows:
Consider the set of at most $\sum_{i=1}^{k+1} m_i$ positions at which some cycle factor
begins in $\alpha_i$.
By splitting all words $\alpha_i$ along these positions we obtain new cycle compressions
of the form $(\beta_{i,1}, \ell_1, \dots, \beta_{i,m}, \ell_m)$
of $\alpha_i$, i.e., the $j$-th cycle factor has uniform length across all $\alpha_i$.
From this representation one easily obtains $m$ instances of  $\MEM(G^\rho_{\ast})$.
\end{proof}
\cref{prop-NP-reduction} yields the $\NP$ upper bound for
\cref{thm-NP}: If $G$ is a finitely generated abelian group, then
$G\cong\Z^n\oplus\bigoplus_{i=1}^m (\Z/r_i\Z)$ for some
$n,r_1,\ldots,r_m\in\N$, so that $\ExpEq(G)$ corresponds to the
solvability problem for linear equation systems over the integers,
possibly with modulo-constraints (if $m>0$). This is a well known
problem in $\NP$.  Moreover, $\MEM(G^\rho_{\ast})$ belongs to $\TC^0$
by \cref{thm-abelian-membership}.

It remains to prove the $\NP$-hardness part of \cref{thm-NP}, which is the content of the next section.

\subsection{NP-hardness}

\begin{thm}
	If $G$ is non-trivial, then $\Knapsack(G \wr \Z)$ is $\NP$-hard.
\end{thm}

\begin{proof}
Since every non-trivial group contains a non-trivial cyclic group, we may assume that $G$ is non-trivial and abelian.
We reduce from 3-dimensional matching, 3DM for short.  
In this problem, we have a set of triples $T =\{e_1,\dots,e_{t}\} \subseteq [1,q] \times [1,q] \times [1,q]$ 
for some $q \geq 1$, and the question whether there is a subset $M \subseteq T$ such that $|M|=q$ and all pairs $(i,j,k),(i^{\prime},j^{\prime},k^{\prime}) \in M$ with $(i,j,k)\neq (i^{\prime},j^{\prime},k^{\prime})$ satisfy $i \neq i^{\prime}, j \neq j^{\prime}$ and $k \neq k^{\prime}$; such a set $M$ is called a matching. 
Since we will write all group operations multiplicatively, we denote the generator of $\Z$ by $a$.

 Let $G$ be a non-trivial group and $g \in G \setminus \{1\}$. We reduce 3DM to $\Knapsack(G \wr \Z)$ in the following way: for every $e_l=(i,j,k) \in T$ let 
\begin{eqnarray*}
w_l & = & a^iga^{q-i+j}ga^{q-j+k}ga^{-2q-k+(3q+1)l}ga^{-(3q+1)l} \\
& = &  \underbrace{a^iga^{q-i+j}ga^{q-j+k}ga^{-2q-k}}_{u_l} \ \underbrace{a^{(3q+1)l}ga^{-(3q+1)l}}_{v_l} .
\end{eqnarray*}
Intuitively, $u_l$ is the word that puts $g$ on positions $i$, $q+j$ and $2q+k$, and 
$v_l$  puts $g$ on position $(3q+1)l$ and then moves the cursor back to $0$. 
Hence,  $v_l$ is contained in $G^{(\Z)}$ and thus commutes with every element of $G\wr\Z$
(recall that $G$ is abelian).

We define the knapsack expression
\[ E = w_1^{x_1}\cdots w_{t}^{x_{t}} \; (a g^{-1})^{3q} a^{-3q}  \;     
\prod_{i=1}^q ( a^{(3q+1)y_i} g^{-1}) \; a^{-(3q+1)y_{q+1}} 
\]
with variables $x_1, \ldots, x_t, y_1, \ldots, y_{q+1}$. 
For all values of these variables, the following equivalences hold.
\begin{eqnarray*}
w_1^{x_1}\cdots w_t^{x_{t}} \; (a g^{-1})^{3q} a^{-3q}  \;     
\prod_{i=1}^q ( a^{(3q+1)y_i} g^{-1}) \; a^{-(3q+1)y_{q+1}} = 1 \ & \Leftrightarrow & \\
u_1^{x_1}\cdots u_t^{x_{t}} \; (a g^{-1})^{3q} a^{-3q}  \;     
v_1^{x_1}\cdots v_t^{x_{t}} \; \prod_{i=1}^q ( a^{(3q+1)y_i} g^{-1}) \; a^{-(3q+1)y_{q+1}} = 1 \ & \Leftrightarrow & \\
\underbrace{u_1^{x_1}\cdots u_t^{x_{t}} \; (a g^{-1})^{3q} a^{-3q}}_{E_1} = 1 \text{ and } 
\underbrace{v_1^{x_1}\cdots v_t^{x_{t}} \; \prod_{i=1}^q ( a^{(3q+1)y_i} g^{-1}) \; a^{-(3q+1)y_{q+1}}}_{E_2} = 1
\end{eqnarray*}
The second equivalence holds because (i)~for all values of the variables, the word $E_1$ 
only affects positions from the interval $[1,3q]$, whereas the word 
$E_2$ 
only affects positions that are multiples of $3q+1$ and (ii)~$E_2$ represents a word in $G^{(\Z)}$.

First assume that there is a matching $M \subseteq T$.
We define a valuation $\nu$  for $E$
by $\nu(x_i)=1$ if $e_i\in M$ and $\nu(x_i)=0$ if $e_i \notin M$. Let $M=\{e_{m_1},\dots,e_{m_q}\}$ such that $m_i < m_j$ for $i<j$
and let $m_0 = 0$.
Then we set $\nu(y_i)=m_{i}-m_{i-1}$ for $1 \leq i \leq q$, and $\nu(y_{q+1})=m_q$. 
Since $M$ is a matching, we have
\[ \nu( u_{e_1}^{x_1}\cdots u_{e_{t}}^{x_{t}} ) =  \prod_{e_l \in M} u_l = (a g)^{3q} a^{-3q}
\]
and thus $\nu(E_1) = 1$.
Furthermore, we have
\[ \nu( v_{e_1}^{x_1}\cdots v_{e_{t}}^{x_{t}} ) = 
\prod_{i=1}^q  a^{(3q+1)m_i}ga^{-(3q+1)m_i} = \prod_{i=1}^q ( a^{(3q+1) (m_{i}-m_{i-1})} g) \; a^{-(3q+1) m_q}
\]
and thus $\nu(E_2) = 1$.

Now assume that there is a valuation $\nu$ for $E$ with $\nu(E_1) = \nu(E_2) = 1$. Let $n_i = \nu(x_i)$ and $m_i = \nu(y_i)$.
For every $1 \leq l \leq t$, we must have $g^{n_l} \in \{1,g\}$, i.e., $n_l \equiv 0 \bmod \mathsf{ord}(g)$ or 
$n_l \equiv 1 \bmod \mathsf{ord}(g)$.
We first show that $q' := \#\{l \mid n_l \equiv 1 \bmod \mathsf{ord}(g) \}=q$. This follows from $\nu(E_2)=1$ and 
the fact that the effect of 
$\prod_{i=1}^q a^{(3q+1)m_i} g^{-1}$ is to multiply the $G$-elements at exactly $q$ many
positions $p$ ($p \equiv 0 \bmod (3q+1)$) with $g^{-1}$. Hence, the effect of 
$v_1^{n_1}\cdots v_t^{n_{t}}$ must be to multiply the $G$-elements at exactly $q$ many
positions $p$ ($p \equiv 0 \bmod (3q+1)$) with $g$. But this means that 
$q'=q$.

So we can assume that $q'=q$. We finally show that $M=\{e_l \mid n_l \equiv 1 \bmod \mathsf{ord}(g) \} \subseteq T$ is a matching: Assume that there are $e=(i,j,k) \in M$ and $e^{\prime}=(i^{\prime},j^{\prime},k^{\prime}) \in M$ with $i=i^{\prime}, j=j^{\prime}$ or $k=k^{\prime}$. Since $q'=q$ this would imply that at most $3q-1$ positions $p$ with $1 \leq p \leq 3q$ can be set to $g$ by the word $u_{e_1}^{n_1}\cdots u_{e_{t}}^{n_{t}}$. 
 But then,  $(a g^{-1})^{3q} a^{-3q}$ would leave a position with value $g^{-1}$, and hence $\nu(E_1) \neq 1$.
 Hence, $M$ must be a matching. Notice that the argumentation of the whole proof still works in the case that we allow the variables $x_1, \dots, x_{t},y_1,\dots, y_{q+1}$ to be integers instead of naturals.
\end{proof}
Note that the above $\NP$-hardness proof also works for the subset sum problem, where the range of the valuation is restricted
to $\{0,1\}$. Moreover, if the word problems for two groups $G$ and $H$ can be solved in polynomial time, then word problem
for $G\wr H$ can be solved in polynomial time as well \cite{MiasnikovVW17}. This implies that subset sum for $G \wr H$ belongs to $\NP$.
Thus, we obtain:

\begin{thm}
         Let $G$ and $H$ be non-trivial finitely generated groups and assume that $H$ contains an element of infinite order.
         Then, the subset sum problem for $G \wr H$ is $\NP$-hard. If moreover, the word problem for 
        $G$ and $H$ can be solved in polynomial time, then  the subset sum problem for $G \wr H$ is $\NP$-complete.
 \end{thm}

\section{Open problems}

Our results yield decidability of $\Knapsack(G\wr H)$ for almost all
groups $G$ and $H$ that are known to satisfy the necessary
conditions. However, we currently have no complete characterization of
those $G$ and $H$ for which $\Knapsack(G\wr H)$ is decidable.
 
Several interesting open problems concerning the complexity of knapsack for wreath products remain. 
We are confident that our $\NP$ upper bound for $\Knapsack(G \wr \Z)$, where $G$ is finitely generated
abelian, can be extended to $\Knapsack(G \wr F)$ for a finitely generated free group $G$ as well as to 
$\Knapsack(G \wr \Z^k)$. Another question is whether the assumption on $G$ being abelian can be weakened.
In particular, we want to investigate whether polynomial time algorithms exist for $\MEM(G^\rho_{\ast})$ for 
certain non-abelian groups $G$. 

The complexity of knapsack for free solvable groups is open as well. Our decidability proof uses the preservation of knapsack-semilinearity
under wreath products (\cref{thm-semilinear2}). Our construction in the proof of \cref{thm-semilinear2} adds for every application of the wreath
product a $\forall^* \exists^*$-quantifier prefix in the formula describing the solution set. Since a free solvable group of class $d$ and rank $r$ is embedded into a $d$-fold 
iterated wreath product of $\Z^r$, this leads to a $\Pi_{2(d-1)}$-formula (for $d=1$, we clearly have a $\Pi_0$-formula). The existence of a solution is then expressed
by a $\Sigma_{2d-1}$-formula. Haase~\cite{Haase14} has shown that the $\Sigma_{i+1}$-fragment of Presburger arithmetic
is complete for the $i$-th level of the so-called weak EXP hierarchy. In addition to the complexity resulting from the quantifier alternations in Presburger arithmetic, our algorithm incurs a doubly exponential increase in the formula size for each application of the wreath product.
This leads to the question whether there is a more efficient
algorithm for knapsack over free solvable groups.

Finally, we are confident that with our techniques from \cite{LohreyZetzsche2016a} one can also show preservation
of knapsack-semilinearity
under graph products.


\appendix

\section{Hyperbolic groups} \label{appendix-hyperbolic}

Let $G$ be a finitely generated group with the finite symmetric generating set $\Sigma$.
The  Cayley-graph of $G$ (with respect to $\Sigma$) is the undirected  graph $\Gamma = \Gamma(G)$ with node set
$G$ and all edges $(g,ga)$ for $g \in G$ and $a \in \Sigma$. We view $\Gamma$ as a geodesic metric space,
where every edge $(g,ga)$ is identified with a unit-length interval. It is convenient to label the directed edge
from $g$ to $ga$ with the generator $a$.
The distance between two points $p,q$ is denoted with $d_\Gamma(p,q)$.
For $g \in G$ let $|g| = d_\Gamma(1,g)$.
For $r \geq 0$,  let
$B_r(1) = \{ g \in G \mid d_\Gamma(1,g) \leq r \}$.

Given a word $w \in \Sigma^*$, one obtains a unique path $P[w]$  that starts in $1$
and is labelled with the word $w$. This path ends 
in the group element represented by $w$. More generally, for $g \in G$ we denote
with $g \cdot P[w]$ the path that starts in $g$  and is labelled with $w$.
We will only consider paths of the form $g \cdot P[w]$.
One views $g \cdot P[w]$ as a continuous mapping from the real interval $[0,|w|]$ to  $\Gamma$.
Such a path  $P : [0,n] \to \Gamma$ is geodesic if $d_\Gamma(P(0),P(n)) = n$; it is
a $(\lambda,\epsilon)$-quasigeodesic if for all points
$p = P(a)$ and $q = P(b)$ we have $|a-b| \leq \lambda \cdot d_\Gamma(p,q) + \varepsilon$. 
We say that a path $P : [0,n] \to \Gamma$ is path from $P(0)$ to $P(n)$.
A word $w \in \Sigma^*$ is geodesic if the path $P[w]$ is geodesic.

A geodesic triangle consists of three points $p,q,r, \in G$ and geodesic paths $P_{p,q}$, $P_{p,r}$, $P_{q,r}$
(the three sides of the triangle),
where $P_{x,y}$ is a path from $x$ to $y$.
For $\delta \geq 0$, the group $G$ is $\delta$-hyperbolic, if for every geodesic triangle,
every point $p$ on one of the three sides has distance at most $\delta$ from a point belonging to one 
of the two sides that are opposite of $p$. Finally, $G$ is hyperbolic, if it is  $\delta$-hyperbolic
for some $\delta \geq 0$. Finitely generated free groups are for instance $0$-hyperbolic.
The property of being hyperbolic is independent of the the chosen generating set. 
The word problem for every hyperbolic group is decidable in linear time. This allows to compute
for a given word $w$ an equivalent geodesic word; the best known algorithm is quadratic. 

Let us fix a $\delta$-hyperbolic group $G$ with the finite symmetric generating set $\Sigma$
for the further discussion.

\begin{lem}[c.f.~\mbox{\cite[8.21]{ghys1990groupes}}] \label{lemma-cyclic-words-quasi-geo}
Let $g \in G$ be of infinite order and let $n \geq 1$. 
Let $u$ be a geodesic word representing $g$.
Then the path  $P[u^n]$ is a $(\lambda,\epsilon)$-quasigeodesic, where 
$\lambda = |g| N$, $\epsilon = 2 |g|^2 N^2 + 2 |g| N$ and $N = |B_{2\delta}(1)|$. 
\end{lem}
Consider two paths $P_1 : [0,n_1] \to \Gamma$, $P_1 : [0,n_2] \to \Gamma$ and let $K$ be a positive
real number.
We say that $P_1$ and $P_2$ asynchronously $K$-fellow travel if there exist two continuous non-decreasing mappings
$\varphi_1 : [0,1] \to [0,n_1]$ and $\varphi_2 : [0,1] \to [0,n_2]$ such that $\varphi_1(0) = \varphi_2(0) = 0$, $\varphi_1(1) = n_1$,
$\varphi_2(1) = n_2$ and for all $0 \leq t \leq 1$, $d_\Gamma(P_1(\varphi_1(t)), P_2(\varphi_2(t))) \leq K$.
Intuitively, this means that one can travel along the paths $P_1$ and $P_2$ asynchronously with variable speeds such that
at any time instant the current points have distance at most $K$.

\begin{lem}[c.f.~\cite{MyNi14}] \label{lemma-asynch-fellow-travel}
Let $P_1$ and $P_2$ be $(\lambda,\epsilon)$-quasigeodesic paths in $\Gamma_G$ and assume that
$P_i$ starts in $g_i$ and ends in $h_i$. Assume that $d_\Gamma(g_1,h_1), d_\Gamma(g_2,h_2) \leq h$.
Then there exists a computable bound 
$K = K(\delta, \lambda,\epsilon, h) \geq h$ such that $P_1$ and $P_2$ 
asynchronously $K$-fellow travel.
\end{lem}

\subsection{Hyperbolic groups are knapsack-semilinear}

In this section, we prove the following result:

\begin{thm} \label{thm-hyperbolic-semilinear}
Every hyperbolic group is knapsack-semilinear.
\end{thm}   
Let us fix a $\delta$-hyperbolic group $G$ and let $\Sigma$ be a finite symmetric generating set for $G$.
We first consider knapsack instances of depth $2$.

\begin{lem} \label{lemma-k=2}
For all $g_1, h_1, g_2, h_2 \in G$ such that 
$g_1$ and $g_2$ have infinite order,
the set $\{ (x_1, x_2) \mid h_1 g_1^{x_1} = g_2^{x_2} h_2 \text{ in } G\}$ is effectively semilinear.
\end{lem}  

\begin{proof}
The semilinear subsets of $\N^k$ are exactly the rational subsets of $\N^k$ \cite{EiSchu69}.
A subset $A \subseteq \N^k$ is rational if it is a homomorphic image of a regular set of words.
In other words, there exists a finite automaton with transitions labeled by elements of $\N^k$
such that $A$ is the set of $v \in \N^k$ that are obtained by summing the transition labels along a 
path from the initial state to a final state. We prove that the set 
$\{ (x_1, x_2) \mid h_1 g_1^{x_1} = g_2^{x_2} h_2 \text{ in } G\}$ is effectively rational.

Let $u_i$ be a geodesic word representing $g_i$ and let $\ell_i = |u_i|$.  Assume that 
$n_1, n_2 \geq 1$ are such that $h_1 g_1^{n_1} = g_2^{n_2} h_2$.
Let $P_1 = h_1 \cdot P[u_1^{n_1}]$ and let 
$P_2 = P[u_2^{n_2}]$.
By \cref{lemma-cyclic-words-quasi-geo}, $P_1$ and $P_2$ are 
$(\lambda,\epsilon)$-quasigeodesics, where $\lambda$ and $\epsilon$
only depend on $\delta$, $|u_1|$ and $|u_2|$.
By \cref{lemma-asynch-fellow-travel}, 
the paths  $P_1$ and $P_2$ asynchronously $K$-fellow travel, where $K$ 
is a computable bound that only depends on $\delta$, $\lambda$, $\epsilon$,
$|g_1|$, $|h_1|$, $|g_2|$, $|h_2|$.
Let $\varphi_1 : [0,1] \to [0, n_1 \cdot \ell_1]$ and $\varphi_2 : [0,1] \to [0, n_2 \cdot  \ell_2]$
be the corresponding continuous non-decreasing  mappings. 

Let $p_{1,i} = h_1 g_1^{i} = P_1(i \cdot \ell_1)$ for $0 \leq i \leq n_1$ and 
$p_{2,j} = g_2^j = P_2(j \cdot \ell_2)$ for $0 \leq j \leq n_2$. Thus, $p_{1,i}$ is a point on $P_1$
and $p_{2,j}$ is a point on $P_2$. 
We define the binary relation $R \subseteq \{ p_{1,i} \mid 0 \leq i \leq n_1 \} \times \{ p_{2,j} \mid 0 \leq j \leq n_2 \}$
by 
$$
R = \{ (p_{1,i}, p_{2,j}) \mid \exists r \in [0,1] :  \varphi_1(r) \in [i \cdot \ell_1, (i+1) \cdot\ell_1), \varphi_2(r) \in [j \cdot \ell_2, (j+1) \cdot \ell_1) \} .
$$
Thus, we take  all pairs $(P_1(\varphi_1(r)), P_2(\varphi_2(r)))$, and push the first (resp., second) point in this pair back along $P_1$ (resp., $P_2$)
to the next point $p_{1,i}$ (resp., $p_{2,j}$). Then $R$ has the following properties:
\begin{itemize}
\item $(0,0), (n_1, n_2) \in R$
\item If $(p_{1,i}, p_{2,j}) \in R$ and $(i,j) \neq (n_1, n_2)$ then one of the following pairs also belongs to $R$: 
$(p_{1,i+1}, p_{2,j})$, $(p_{1,i}, p_{2,j+1})$, $(p_{1,i+1}, p_{2,j+1})$.
\item If $(p_{1,i}, p_{2,j}) \in R$, then $d_\Gamma(p_{1,i}, p_{2,j}) \leq K + \ell_1 + \ell_2$.
\end{itemize}
Let $r = K + \ell_1 + \ell_2$.
We can now construct a finite automaton over $\mathbb{N} \times \mathbb{N}$ that accepts
the set $\{ (x_1, x_2) \mid h_1 g_1^{x_1} = g_2^{x_2} h_2 \text{ in } G\}$. 
The set of states consists of $B_r(1)$. The initial state is $h_1$, the final state is $h_2$.
Finally, the transitions are the following:
\begin{itemize}
\item $p \xrightarrow{(0,1)} q$ for $p,q \in B_r(1)$ if $p = g_2 q$
\item $p \xrightarrow{(1,0)} q$ for $p,q \in B_r(1)$ if $p g_1 = q$
\item $p \xrightarrow{(1,1)} q$ for $p,q \in B_r(1)$ if $p g_1 = g_2 q$
\end{itemize}
By the above consideration, it is clear that this automaton accepts the set 
$\{ (x_1, x_2) \mid h_1 g_1^{x_1} = g_2^{x_2} h_2 \text{ in } G\}$.
\end{proof}
We can now prove \cref{thm-hyperbolic-semilinear}.

\begin{proof}[Proof of \cref{thm-hyperbolic-semilinear}]
Consider a knapsack expression $E = v_1 u_1^{x_1} v_2 u_2^{x_2} v_3 \cdots u_{k}^{x_{k}} v_{k+1}$.
We want to show that the set of all  solutions of $E=1$ is a semilinear subset of $\mathbb{N}^k$. 
For this we construct a Presburger formula with free variables $x_1, \ldots, x_k$
that is equivalent to $E=1$. We do this by induction on the depth $k$. Therefore, we can use 
in our Presburger formula also knapsack equations of the form $F=1$, where $F$ has depth
at most $k-1$.

Let $g_i \in G$ be the group element represented by the word $u_i$. 
In a hyperbolic group the order of torsion elements is bounded by a fixed constant
that only depends on the group, see also the proof of \cite[Theorem~6.7]{MyNiUs14}).
This allows to check for each $g_i$ whether it has finite order, and to compute the order
in the positive case. Assume that $g_i$ has finite order $m_i$.
We then produce for every number $0 \leq d \leq m_i-1$ a knapsack instance
of depth $k-1$ by replacing $u_i^{x_i}$ by $u_i^d$, which by induction can be transformed
into an equivalent Presburger formula. We then take the disjunction of all these
Presburger formulae for all  $0 \leq d \leq m_i-1$. A similar argument shows that it suffices to construct 
a Presburger formula describing all solutions in $\mathbb{N}_+^k$ (where $\mathbb{N}_+ = \mathbb{N} \setminus \{0\}$).

By the above discussion, we 
can assume that all $u_i$ represent group elements of infinite order.
The case that $k \leq 2$ is covered by \cref{lemma-k=2}. Hence, we assume that
$k \geq 3$. By the above remark, we only need to consider
valuations $\nu$ such that $\nu(x_i)>0$ for all $i \in [1,k]$.  
Moreover, we can assume that $E$ has the form $u_1^{x_1} \cdots u_{k}^{x_{k}} v$,
where all $u_i$ and $v$ are geodesic words.
By \cref{lemma-cyclic-words-quasi-geo}
for every valuation $\nu$, all words $u_i^{\nu(x_i)}$ are $(\lambda,\epsilon)$-quasigeodesics
for certain constants $\lambda$ and $\epsilon$.

Consider a solution $\nu$ and let $n_i = \nu(x_i)$ for $i \in [1,k]$. 
Consider the polygon obtained by traversing the closed path labelled with 
$u_1^{x_1} \cdots u_{k}^{x_{k}} v$. We partition this path into segements 
$P_1, \ldots, P_k,Q$, where 
$P_i$ is the subpath labelled with $u_i^{n_i}$ and $Q$ is the subpath labelled with $v$.
We consider these subpaths as the sides of a $(k+1)$-gon, see \cref{polygon}.
Since all sides of this $(k+1)$-gon are $(\lambda,\epsilon)$-quasigeodesics,
we can apply \cite[Lemma 6.4]{MyNiUs14}: Every side of the $(k+1)$-gon is contained in the 
$h$-neighborhoods of the other sides, where $h = (\kappa+\kappa\log(k+1))$ for a constant $\kappa$ that only depends
on the constants $\delta, \lambda, \varepsilon$. 

\newlength{\R}\setlength{\R}{2.7cm}
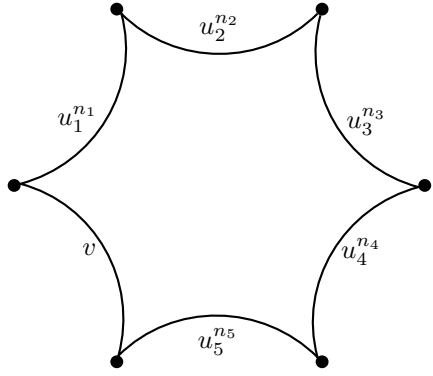
\begin{figure}[t]
\centering
\begin{tikzpicture}
  [inner sep=.5mm,
  minicirc/.style={circle,draw=black,fill=black,thick}]

  \node (circ1) at ( 60:\R) [minicirc] {};
  \node (circ2) at (120:\R) [minicirc] {};
  \node (circ3) at (180:\R) [minicirc] {};
  \node (circ4) at (240:\R) [minicirc] {};
  \node (circ5) at (300:\R) [minicirc] {};
  \node (circ6) at (360:\R) [minicirc] {};

  \draw [thick] (circ1) to [out=225, in=-45] node[above=1mm] {$u_2^{n_2}$} (circ2) to [out=285, in=15] node[above=0.5mm,left=0.5mm] {$u_1^{n_1}$} (circ3) 
  to [out=345, in=75] node[below=0.5mm,left=0.5mm] {$v$} (circ4) to [out=45, in=135] node[below=1mm] {$u_5^{n_5}$} (circ5) to [out=105, in=195] node[below=0.5mm,right=0.5mm] {$u_4^{n_4}$} (circ6) to [out=165, in=255] node[right=1mm] {$u_3^{n_3}$} (circ1);
\end{tikzpicture}
\caption{\label{polygon} The $(k+1)$-gon for $k = 5$ from the proof of \cref{thm-hyperbolic-semilinear}}
\end{figure}

Let us now consider the side $P_2$ of the quasigeodesic $(k+1)$-gon. It is labelled
with $u_2^{x_2}$. Its neighboring  sides are $P_1$ and $P_3$ (recall that $k \geq 3$) 
and are labelled with $u_1^{x_1}$ and $u_3^{x_3}$.\footnote{We take the side $P_2$ since
$Q$ is not a neighboring side of $P_2$. This avoids some additional cases in the following case distinction.}
We now distinguish the following cases. In each case we cut the $(k+1)$-gon into smaller 
pieces along paths of length $\leq h$, and these smaller pieces will correpsond to knapsack
instances of smaller depth. When we speak of a point on the $(k+1)$-gon, we mean a node of the Cayley graph
(i.e., an element of the group $G$) and not a point in the interior of an edge.
Moreover, when we peak of the successor point of a point $p$, we refer to the clockwise order on the 
$(k+1)$-gon, where the sides are traversed in the order $P_1, \ldots, P_k,Q$.

\medskip
\noindent
{\em Case 1:} There is a point on $p \in P_2$ that has distance at most $h$ from a node $q \in P_4 \cdots P_k$. Let us assume that 
$q \in P_i$ where $i \in [4,k]$.
We now construct two new knapsack instances $F_t$ and $G_t$ for all words $w \in \Sigma^*$ of length at most $h$
and all factorizations $u_2 = u_{2,1} u_{2,2}$ and $u_i = u_{i,1} u_{i,2}$, where $t = (i,w,u_{2,1},u_{2,2},u_{i,1},u_{i,2})$:
\begin{eqnarray*}
F_t &=& u_1^{x_1} u_2^{y_2} (u_{2,1} w u_{i,2}) u_i^{z_i} u_{i+1}^{x_{i+1}} \cdots u_k^{x_k} v \ \text{ and } \\
G_t  &=& u_{2,2} u_2^{z_2} u_3^{x_3} \cdots u_{i-1}^{x_{i-1}} u_i^{y_i} (u_{i,1} w^{-1})
\end{eqnarray*}
Here $y_2,z_2,y_i,z_i$ are new variables.
The situation looks as follows, where the case $i=k=5$ is shown:
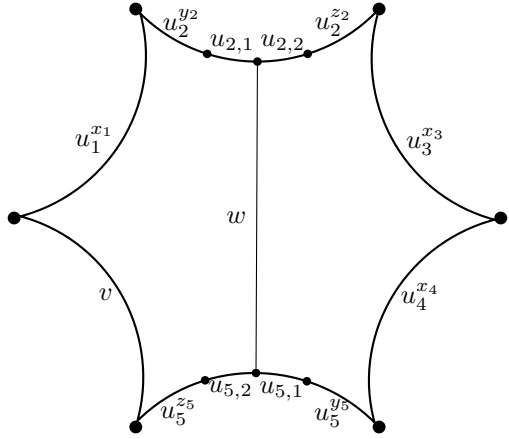
\newlength{\Q}\setlength{\Q}{3.2cm}
\begin{center}
\begin{tikzpicture}
  [inner sep=.5mm,
  minicirc/.style={circle,draw=black,fill=black,thick}]

 \tikzstyle{small} = [circle,draw=black,fill=black,inner sep=.3mm]
  \node (circ1) at ( 60:\Q) [minicirc] {};
  \node (circ2) at (120:\Q) [minicirc] {};
  \node (circ3) at (180:\Q) [minicirc] {};
  \node (circ4) at (240:\Q) [minicirc] {};
  \node (circ5) at (300:\Q) [minicirc] {};
  \node (circ6) at (360:\Q) [minicirc] {};

  \draw [thick] (circ1) to [out=225, in=-45] node[pos=.3,small] {} node[pos=.5,small]  (a) {} node[pos=.7,small] {} 
     node[above=0.2mm,pos=.4] {$u_{2,2}$} node[above=0.2mm,pos=.6] {$u_{2,1}$} node[pos=0.05,above=1mm,left=1mm] {$u_2^{z_2}$}
     node[pos=0.95,above=1mm,right=1mm] {$u_2^{y_2}$}
     (circ2) to [out=285, in=15] node[above=0.5mm,left=0.5mm] {$u_1^{x_1}$} (circ3) 
  to [out=345, in=75] node[below=0.5mm,left=0.5mm] {$v$} (circ4) to [out=45, in=135] 
   node[pos=.3,small] {} node[pos=.5,small]  (b){} node[pos=.7,small] {} 
     node[below=0.2mm,pos=.4] {$u_{5,2}$} node[below=0.2mm,pos=.6] {$u_{5,1}$} node[pos=0.05,below=1mm,right=1mm] {$u_5^{z_5}$}
     node[pos=0.95,below=1mm,left=1mm] {$u_5^{y_5}$}  (circ5) to [out=105, in=195] node[below=0.5mm,right=0.5mm] {$u_4^{x_4}$} (circ6) to [out=165, in=255] node[right=1mm] {$u_3^{x_3}$} (circ1);
     
     \draw (a) to node[left=.8mm] {$w$} (b);
\end{tikzpicture}
\end{center}
Note that $F_t$ and $G_t$ have depth at most $k-1$. Lets say that a tuple
$t = (i,w,u_{2,1},u_{2,2},u_{i,1},u_{i,2})$ is valid for case 1 if 
$i  \in [4,k]$, $w \in \Sigma^*$, $|w| \leq h$,
$u_2 = u_{2,1} u_{2,2}$ and $u_i = u_{i,1} u_{i,2}$.
Moreover, let $A_1$ be the following formula, where $t$ ranges over all tuples that are valid for case 1,
and $i$ is the first component of the tuple $t$:
\[
A_1 =  \bigvee_t \exists y_2,z_2,y_i,z_i : x_2 = y_2 + 1 + z_2 \wedge x_i  = y_i + 1 + z_i \wedge F_t = 1 \wedge G_t = 1
\]
{\em Case 2:} There is a point on $p \in P_2$ that has distance at most $h$ from a node $q \in Q$.
We construct two new knapsack instances $F_t$ and $G_t$ for all words $w \in \Sigma^*$ of length at most $h$
and all factorizations $u_2 = u_{2,1} u_{2,2}$ and $v = v_1 v_2$, where $t = (w,u_{2,1},u_{2,2},v_1,v_2)$:
\begin{eqnarray*}
F_t &=& u_1^{x_1} u_2^{y_2} (u_{2,1} w v_2)  \ \text{ and } \\
G_t  &=& u_{2,2} u_2^{z_2} u_3^{x_3} \cdots u_k^{x_k} (v_1 w^{-1})
\end{eqnarray*}
As in case 1, $y_2,z_2$ are new variables and 
$F_t$ and $G_t$ have depth at most $k-1$. The situation looks as follows:
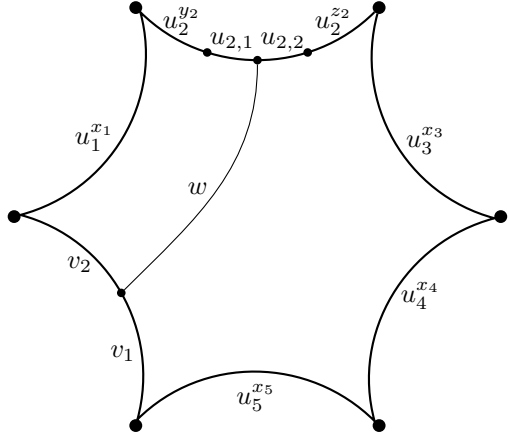
\begin{center}
\begin{tikzpicture}
  [inner sep=.5mm,
  minicirc/.style={circle,draw=black,fill=black,thick}]

 \tikzstyle{small} = [circle,draw=black,fill=black,inner sep=.3mm]
  \node (circ1) at ( 60:\Q) [minicirc] {};
  \node (circ2) at (120:\Q) [minicirc] {};
  \node (circ3) at (180:\Q) [minicirc] {};
  \node (circ4) at (240:\Q) [minicirc] {};
  \node (circ5) at (300:\Q) [minicirc] {};
  \node (circ6) at (360:\Q) [minicirc] {};

  \draw [thick] (circ1) to [out=225, in=-45] node[pos=.3,small] {} node[pos=.5,small]  (a) {} node[pos=.7,small] {} 
     node[above=0.2mm,pos=.4] {$u_{2,2}$} node[above=0.2mm,pos=.6] {$u_{2,1}$} node[pos=0.05,above=1mm,left=1mm] {$u_2^{z_2}$}
     node[pos=0.95,above=1mm,right=1mm] {$u_2^{y_2}$}
     (circ2) to [out=285, in=15] node[above=0.5mm,left=0.5mm] {$u_1^{x_1}$} (circ3) 
  to [out=345, in=75] node[below=1mm,pos=.25] {$v_2$} node[left=.2mm,pos=.75] {$v_1$}  node[pos=.5,small]  (b){} (circ4) to [out=45, in=135]  
  node[below=1mm] {$u_5^{x_5}$}
    (circ5) to [out=105, in=195] node[below=0.5mm,right=0.5mm] {$u_4^{x_4}$} (circ6) to [out=165, in=255] node[right=1mm] {$u_3^{x_3}$} (circ1);
     
     \draw (a) to [out=270, in=45] node[left=.8mm] {$w$} (b);
\end{tikzpicture}
\end{center}
We say that a tuple
$t = (w,u_{2,1},u_{2,2},v_1,v_2)$ is valid for case 2 if 
$w \in \Sigma^*$, $|w| \leq h$,
$u_2 = u_{2,1} u_{2,2}$ and $v = v_1 v_2$.
Moreover, let $A_2$ be the following formula, where $t$ ranges over all tuples that are valid for case 2:
\[
A_2 =  \bigvee_t \exists y_2,z_2 : x_2 = y_2 + 1 + z_2 \wedge F_t = 1 \wedge G_t = 1
\]
{\em Case 3:} Every point $p \in P_2$ has distance at most $h$ from a point on $P_1$.
Let $q$ be the unique point in $P_2 \cap P_3$ and let $p \in P_1$ be a point with $d_\Gamma(p,q) \leq h$.
We construct two new knapsack instances $F_t$ and $G_t$ for all words $w \in \Sigma^*$ of length at most $h$
and all factorizations $u_1 = u_{1,1} u_{1,2}$, where $t = (w,u_{1,1},u_{1,2})$:
\begin{eqnarray*}
F_t &=& u_1^{y_1} (u_{1,1} w) u_3^{x_3} \cdots u_k^{x_k} v  \ \text{ and } \\
G_t  &=& u_{1,2} u_1^{z_1} u_2^{x_2} w^{-1}
\end{eqnarray*}
Since $k \geq 3$,  $F_t$ and $G_t$ have depth at most $k-1$. 
The situation looks as follows:
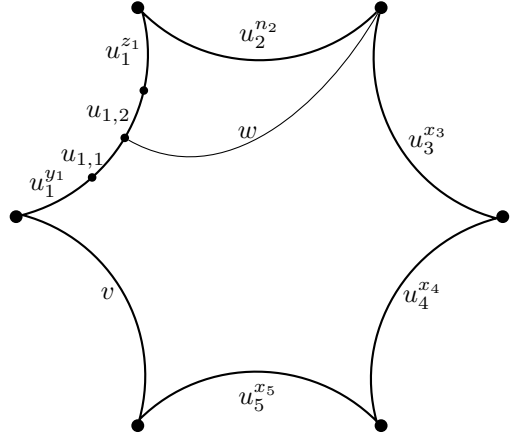
\begin{center}
\begin{tikzpicture}
  [inner sep=.5mm,
  minicirc/.style={circle,draw=black,fill=black,thick}]

 \tikzstyle{small} = [circle,draw=black,fill=black,inner sep=.3mm]
  \node (circ1) at ( 60:\Q) [minicirc] (b) {};
  \node (circ2) at (120:\Q) [minicirc] {};
  \node (circ3) at (180:\Q) [minicirc] {};
  \node (circ4) at (240:\Q) [minicirc] {};
  \node (circ5) at (300:\Q) [minicirc] {};
  \node (circ6) at (360:\Q) [minicirc] {};

    \draw [thick] (circ1) to [out=225, in=-45] node[above=1mm] {$u_2^{n_2}$} (circ2) to [out=285, in=15] node[pos=.3,small] {} node[pos=.5,small]  (a) {}   node[pos=.7,small] {} 
     node[above=1mm,left=.1mm,pos=.4] {$u_{1,2}$} node[above=2mm,left=0mm,pos=.6] {$u_{1,1}$} node[pos=0.15,left=.1mm] {$u_1^{z_1}$}
     node[pos=0.9,above=.5mm] {$u_1^{y_1}$} (circ3) 
  to [out=345, in=75] node[below=0.5mm,left=0.5mm] {$v$} (circ4) to [out=45, in=135] node[below=1mm] {$u_5^{x_5}$} (circ5) to [out=105, in=195] node[below=0.5mm,right=0.5mm] {$u_4^{x_4}$} (circ6) to [out=165, in=255] node[right=1mm] {$u_3^{x_3}$} (circ1);
     
     \draw (a) to [out=-30, in=-120] node[left=.8mm] {$w$} (b);
\end{tikzpicture}
\end{center}
We say that a triple
$t = (w,u_{1,1},u_{1,2})$ is valid for case 3 if 
$w \in \Sigma^*$, $|w| \leq h$ and $u_1 = u_{1,1} u_{1,2}$.
Moreover, let $A_3$ be the following formula, where $t$ ranges over all tuples that are valid for case 3:
\[
A_3 =  \bigvee_t \exists y_1,z_1 : x_1 = y_1 + 1 + z_1 \wedge F_t = 1 \wedge G_t = 1
\]
{\em Case 4:} Every point $p \in P_2$ has distance at most $h$ from a point on $P_3$. This case is of course
completely analogous to case 3 and yields a corresponding formula $A_4$.

\medskip
\noindent
{\em Case 5:} Every point $p \in P_2$ has distance at most $h$ from a point on $P_1 \cup P_3$
but $P_2$ is neither contained in the $h$-neighborhood of $P_1$ nor in the $h$-neighborhood of $P_3$.
Hence there exists points $p_1, p_3 \in P_2$ which are connected by an edge and such that $p_1$ has 
distance at most $h$ from $P_1$ and $p_3$ has distance at most $h$ from $P_3$. Therefore, $p_1$ has
distance at most $h+1$ from $P_1$ as well as distance at most $h+1$ from $P_3$.
We construct three new knapsack instances $F_t$, $G_t$, $H_t$ for all words $w_1,w_2 \in \Sigma^*$ with 
$|w_1|, |w_2| \leq h+1$
and all factorizations $u_1 = u_{1,1} u_{1,2}$, $u_2 = u_{2,1} u_{2,2}$, and 
$u_3 = u_{3,1} u_{3,2}$, where $t = (w_1,w_2,u_{1,1},u_{1,2},u_{2,1},u_{2,2},u_{3,1},u_{3,2})$:
\begin{eqnarray*}
F_t &=& u_1^{y_1} (u_{1,1} w_1 w_2 u_{3,2}) u_3^{z_3} u_4^{x_4} \cdots u_k^{x_k} v, \\
G_t &=& u_{1,2} u_1^{z_1} u_2^{y_2} u_{2,1} w_1^{-1}, \\
H_t &=& u_{2,2} u_2^{z_2} u_3^{y_3} u_{3,1} w_2^{-1}
\end{eqnarray*}
Since $k \geq 3$,  $F_t$, $G_t$ and $H_t$ have depth at most $k-1$. 
The situation looks as follows:
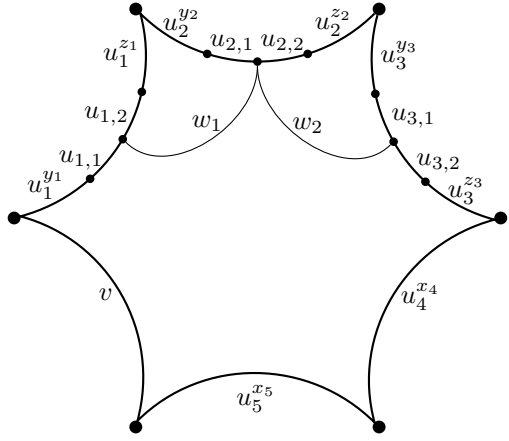
\begin{center}
\begin{tikzpicture}
  [inner sep=.5mm,
  minicirc/.style={circle,draw=black,fill=black,thick}]

 \tikzstyle{small} = [circle,draw=black,fill=black,inner sep=.3mm]
  \node (circ1) at ( 60:\Q) [minicirc]  {};
  \node (circ2) at (120:\Q) [minicirc] {};
  \node (circ3) at (180:\Q) [minicirc] {};
  \node (circ4) at (240:\Q) [minicirc] {};
  \node (circ5) at (300:\Q) [minicirc] {};
  \node (circ6) at (360:\Q) [minicirc] {};

    \draw [thick] (circ1) to [out=225, in=-45]  node[pos=.3,small] {} node[pos=.5,small]  (a) {} node[pos=.7,small] {} 
     node[above=0.2mm,pos=.4] {$u_{2,2}$} node[above=0.2mm,pos=.6] {$u_{2,1}$} node[pos=0.05,above=1mm,left=1mm] {$u_2^{z_2}$}
     node[pos=0.95,above=1mm,right=1mm] {$u_2^{y_2}$} (circ2) to [out=285, in=15] node[pos=.3,small] {} node[pos=.5,small]  (b) {}   node[pos=.7,small] {} 
     node[above=1mm,left=.1mm,pos=.4] {$u_{1,2}$} node[above=2mm,left=0mm,pos=.6] {$u_{1,1}$} node[pos=0.15,left=.1mm] {$u_1^{z_1}$}
     node[pos=0.9,above=.5mm] {$u_1^{y_1}$} (circ3) 
  to [out=345, in=75] node[below=0.5mm,left=0.5mm] {$v$} (circ4) to [out=45, in=135] node[below=1mm] {$u_5^{x_5}$} (circ5) to [out=105, in=195] node[below=0.5mm,right=0.5mm] {$u_4^{x_4}$} (circ6) to [out=165, in=255] node[pos=.3,small] {} node[pos=.5,small]  (c) {}   node[pos=.7,small] {} 
     node[pos=.4,right=.5mm] {$u_{3,2}$} node[pos=.6,right=.5mm] {$u_{3,1}$} node[pos=0.25,right=1mm] {$u_3^{z_3}$}
     node[pos=0.85,right=.5mm] {$u_3^{y_3}$} (circ1);
     
     \draw (a) to [out=-90, in=-45] node[above=.8mm] {$w_1$} (b);
      \draw (a) to [out=-90, in=-135] node[above=.8mm] {$w_2$} (c);
\end{tikzpicture}
\end{center}
We say that a tuple
$t = (w_1,w_2,u_{1,1},u_{1,2},u_{2,1},u_{2,2},u_{3,1},u_{3,2})$ is valid for case 5 if 
$w_1,w_2 \in \Sigma^*$,
$|w_1|,|w_2| \leq h+1$, $u_1 = u_{1,1} u_{1,2}$, $u_2 = u_{2,1} u_{2,2}$, and 
$u_3 = u_{3,1} u_{3,2}$.
Moreover, let $A_5$ be the following formula, where $t$ ranges over all tuples that are valid for case 5:
\begin{align*}
A_5 =  \bigvee_t \exists y_1,z_1,y_2,z_2,y_3,z_3 : \ & x_1 = y_1 + 1 + z_1 \wedge x_2 = y_2 + 1 + z_2 \wedge x_3 = y_3 + 1 + z_3 \; \wedge \\
& F_t = 1 \wedge G_t = 1  \wedge H_t = 1 .
\end{align*}
Our final formula is $A_1 \vee A_2 \vee A_3 \vee A_4 \vee A_5$. It is easy to check that a valuation 
$\nu : \{ x_1, \ldots, x_k \}$ satisfies $\nu(E)=1$ if and only if $\nu$ makes $A_1 \vee A_2 \vee A_3 \vee A_4 \vee A_5$ true.
If $\nu(E)=1$ holds, then one of the above five cases holds, in which case $\nu$ makes the corresponding
formula $A_i$ true. Vice versa, if $\nu$ makes one of the formulas $A_i$ true then $\nu(E)=1$ holds.
\end{proof}

\end{document}